\documentclass{article}%
\usepackage{amsmath}
\usepackage{amsfonts}
\usepackage{amssymb}
\usepackage{graphicx}%
\providecommand{\U}[1]{\protect\rule{.1in}{.1in}}
\newtheorem{theorem}{Theorem}

\newtheorem{definition}[theorem]{Definition}

\newtheorem{lemma}[theorem]{Lemma}

\newtheorem{proposition}[theorem]{Proposition}
\newtheorem{remark}[theorem]{Remark}

\newenvironment{proof}[1][Proof]{\noindent\textbf{#1.} }{\ \rule{0.5em}{0.5em}}
\begin{document}

\title{On a free boundary problem and minimal surfaces}
\author{Yong Liu\\School of Mathematics and Physics, \\North China Electric Power University, Beijing, China,\\Email: liuyong@ncepu.edu.cn
\and Kelei Wang\\School of Mathematics and Statistics, Wuhan Univeristy\\Email:wangkelei@whu.edu.cn
\and Juncheng Wei\\Department of Mathematics, \\University of British Columbia, Vancouver, B.C., Canada, V6T 1Z2\\Email: jcwei@math.ubc.ca}
\maketitle

\begin{abstract}
From minimal surfaces such as Simons' cone and catenoids, using
refined Lyapunov-Schmidt reduction method, we construct new solutions for a free
boundary problem whose free boundary has two components. In dimension $8$, using
variational arguments, we also obtain solutions which are global  minimizers of the
corresponding energy functional. This shows that Savin's theorem  \cite{S} is optimal.

\end{abstract}

\section{\bigskip Introduction}
In this paper, we are interested in  the following free
boundary problem$:$
\begin{equation}
\left\{
\begin{array}
[c]{l}%
\Delta u=0\text{ in }\Omega:=\left\{  -1<u<1\right\}  ,\\
\left\vert \nabla u\right\vert =1\text{ on }\partial\Omega.
\end{array}
\right.  \label{FB}%
\end{equation}
Here the domain $\Omega\subset\mathbb{R}^{n}$ is a priori unspecified and
$\partial\Omega$ is the free boundary.  Solutions of $\left(  \ref{FB}\right)
$ arise formally as critical points of the energy functional:%
\begin{equation}
J\left(  u\right)  :=\int\left[  \left\vert \nabla u\right\vert ^{2}%
+\chi_{\left(  -1,1\right)  }\left(  u\right)  \right]  .
\label{two phase problem}%
\end{equation}
In this variational formulation, the boundary condition $\left\vert \nabla
u\right\vert =1$ should be understood in some weak sense if the free boundary
$\partial\Omega$ is not regular enough. Problem $\left(  \ref{FB}\right)
$ can be regarded as a simplified version of the classical one-phase free
boundary problem:%
\begin{equation}
\left\{
\begin{array}
[c]{l}%
\Delta u=0\text{ in }\Omega:=\left\{  u>0\right\}  ,\\
\left\vert \nabla u\right\vert =1\text{ on }\partial\Omega.
\end{array}
\right.  \label{onephase}%
\end{equation}

\medskip

The regularity of the free boundary problems  actually has been a subject of extensive
studies, pioneered by the work of Caffarelli (see \cite{Alt,C2,Ca1,Ca2,Ca3} and the references therein). It is now known that in dimension $n\leq4$, the free boundary of a
solution to $\left(  \ref{onephase}\right)  $ has no singularity, provided
that it is an energy \textit{minimizer} (\cite{Alt,C1,Jerison1}). In fact, it
is conjectured that for $n\leq6,$ minimizers should have smooth free boundary.
On the other hand, in higher dimensions $n\geq7,$ an energy minimizing free
boundary may have singularities. To explain this, let us mention that by blow
up analysis, the regularity of the free boundary is essentially related to the
existence or nonexistence of minimizing cone. Let us consider the cone in
$\mathbb{R}^{n}$ given by (see \cite{C1})
\begin{equation}
\left\vert x_{n}\right\vert <\alpha_{n}\sqrt{x_{1}^{2}+...+x_{n}^{2}},
\label{cone}%
\end{equation}
where $\alpha_{n}$ is a dimensional constant choosen such that on this cone
there is a solution to the one-phase free boundary problem. It has been proved
(\cite{Jerison2}) that in dimension $n=7$ (Actually also for $n=9,11,13,15$
and hopefully for all $n\geq7$), the solution to $\left(  \ref{onephase}%
\right)  $ corresponding to the cone $\left(  \ref{cone}\right)  $ is a
minimizer for the energy functional. For $3\leq n\leq6,$ this solution is
already known to be unstable, thanks to the work of \cite{C1}. On the other
hand, if a solution to $\left(  \ref{onephase}\right)  $ is a minimizer and if
the free boundary is a priori a graph, then by the result of \cite{Jerison3},
this free boundary will be real analytic. It is worth pointing out that all
these regularity results are in many respect analogous to that of the minimal
surface theory, and these two subjects are closely related.

\medskip

In $\mathbb{R}^{2}$, Traizet (\cite{Traizet,Traizet2}) proved that there is a
one-to-one correspondence between solutions to $\left(  \ref{FB}\right)  $ and
$\left(  \ref{onephase}\right)  $ and certain type of minimal surfaces in
$\mathbb{R}^{3}.$ Hence at least in dimension two$,$ this problem is well
understood, although even for the minimal surfaces in $\mathbb{R}^{3},$ many
questions remain unanswered up to now. We also refer to \cite{Pacard},
\cite{Jerison5}, \cite{Ruiz} for related existence and classification results
for other types of free boundary problems. Now we emphasize that in higher
dimensions, the explicit correspondence between minimal surfaces and free
boundary problem is not available. However, in $\mathbb{R}^{9},$ it is proved
by Kamburov (\cite{Kam}) using sub and super solution method that there exists
a solution to $\left(  \ref{FB}\right)  $ where the free boundary is close to
two copies of the famous Bomberi-De Giorgi-Giusti minimal graph. His result
indicates that there should be deeper relation between minimal surface and the
free boundary problem (\ref{FB}). Here in this paper we would like to further explore
this relation by constructing solutions to $\left(  \ref{FB}\right)  $ based
on minimal surfaces.

\medskip

Notice that problem (\ref{FB}) can be considered as special case of  over-determined problems. In recent years the following so-called Serrin's overdetermined problem
\begin{equation}
\left\{
\begin{array}
[c]{l}%
\Delta u=f(u)\text{ in }\Omega:=\left\{  u>0\right\}  ,\\
u=0, \ \left\vert \nabla u\right\vert = Constant \text{ on }\partial\Omega.
\end{array}
\right.  \label{onephase}%
\end{equation}
has also received much attention.  We refer to \cite{F1, FV0, FMV, Manuel2, RS, Ruiz, Si, WW} and the references therein.

\medskip

Another motivation  for studying (\ref{FB}) is related to De Giorgi's conjecture. In 1978 De Giorgi conjectured that the only bounded  solution to
\begin{equation}
\label{AC}
\Delta u+ u- u^3=0 \ \mbox{in} \ {\mathbb{R}}^n
\end{equation}
 which is monotone in $x_n$  must be one dimensional (up to rotation and translation) at least in dimension $n\leq 8$. De Giorgi's conjecture
is a natural, parallel statement to Bernstein theorem for minimal graphs, which in its most general
form, due to Simons \cite{Simons}, states that any minimal hypersurface in $\mathbb{R}^n$, which is also a graph of a
function of $n-1$ variables, must be a hyperplane if $n \leq  8$. Strikingly, Bombieri, De Giorgi and
Giusti \cite{BDG} proved that this fact is false in dimension $n \geq 9$.

Great advance in De Giorgi's conjecture has been achieved in recent years, having been fully
established in dimensions $n = 2$ by Ghoussoub and Gui \cite{GG} and for $n = 3$ by Ambrosio and
Cabr´e \cite{AC}. A celebrated result by Savin \cite{Savin1} established its validity for $4 \leq n \leq 8$ under the following additional assumption
$$ \lim_{x_n \to \pm \infty} u(x^{'}, x_n)=\pm 1$$
(See Savin-Sciunzi-Valdinoci \cite{VSS} and  Farina-Valdinoci \cite{FV1, FV2} for generalizations.) Del Pino, Kowalczyk and Wei \cite{DKW} constructed a counterexample in dimensions $n\geq 9$.

\medskip

Replacing the monotonicity assumption by global minimality of energy, Savin proved that in dimensions $n\leq 7$ all global minimizers  to (\ref{AC}) are one-dimensional. We proved that Savin's result is optimal by constructing global minimizers in dimensional $8$ (\cite{Liu}). (Stable solutions are constructed in Pacard-Wei \cite{PW}.)

\medskip

In a recent paper \cite{S},  Savin also extended the De Giorgi type conjecture result to problems with more general nonlinearities including
$$ \Delta u= W_u (u)$$
where $W= (1-u^2)^{\alpha}, \alpha \geq 0$. $\alpha =0$ corresponds to  the problem (\ref{FB}). In particular he proved  global minimizers of (\ref{FB}) must be one-dimensional if $n\leq 7$.  One of our results below shows that this is optimal.

\medskip

The purpose of this paper is to establish the connection between minimal surfaces and problem (\ref{FB}). In particular we  shall construct new solutions to (\ref{FB}) by developing new gluing methods for overdetermined problems.  We know very little information about the solutions of (\ref{FB}) in dimensions $n\geq 3$. In dimension $2$ Traizet's characterization \cite{Traizet} reduces the problem to singly minimal surfaces in $\mathbb{R}^3$. In dimension $9$, Kambrunov's solution \cite{Kam}  is a monotone solution whose two components are approximately Bombieri-De Giorgi-Giusti graphs.  For $ 3\leq n \leq 8$ we know no solutions to (\ref{FB}). In this paper we establish a connection between minimal surfaces and solutions to (\ref{FB}) and thereby provide plenty of new solutions to (\ref{FB}). In addition, we shall prove the existence of global minimizers in $\mathbb{R}^8$ and execute the Jeriosn-Monneau program for problem (\ref{FB}).

\medskip

Rather than considering the most general minimal surfaces, we shall focus on two types of
classical minimal surfaces. The first type of minimal surfaces are the area
minimizing cones (minimizing hypersurfaces) in $\mathbb{R}^{n}$($n\geq8$). As
an example, let us consider the famous Simons' cone:%
\[
S:=\left\{  \left(  x_{1},...,x_{8}\right)  \in\mathbb{R}^{8}:\Sigma_{i=1}%
^{4}x_{i}^{2}=\Sigma_{i=5}^{8}x_{i}^{2}\right\}  .
\]
This is a minimal surface with one singularity at the origin. The fact that
Simons' cone is area minimizing has been proved in the classical work of
Bombieri-De Giorgi-Giusti \cite{BDG}. Using the minimizing property,
Hardt-Simon \cite{HS} was able to show that there exists a family of foliated
minimal surfaces $S_{\delta}^{+}$ lying on one side of the cone and is
asymptotic to the cone at infinity. Similarly, the other side of the cone is
also foliated by a family of minimal surfaces $S_{\delta}^{-}.$ Due to scaling
invariance, this family of surfaces $S_{\delta}^{\pm}$ can be obtained simply
as homothety of $S_{1}^{\pm},$ that is $S_{\delta}^{\pm}=\delta^{-1}S_{1}%
^{\pm}.$ Actually, Hardt-Simon proved more. They showed that the Simons' cone
is strictly area minimizing which implies that each surface $S_{\delta}^{\pm}$
approaches the cone at the slowest possible rate.

As we mentioned before, there should be similarities between the minimal
surfaces and free boundary problem. A natural question is whether there are
analogous solutions for the free boundary problem $\left(  \ref{FB}\right)  $
as the Simons' cone and its associated foliation. We answer this affirmatively.

\begin{theorem}
\label{simons}For each $\varepsilon$ small enough, there exists domain
$\Omega^{\varepsilon}$ close to the radius one tubular neighbourhood of
$S_{\varepsilon}^{+}$ and solution $u_{\varepsilon}$ to the free boundary
problem $\left(  \ref{FB}\right)  .$ Moreover, $u_{\varepsilon}$ is stable in
the sense that there exists a function $\Phi>0$ in $\Omega^{\varepsilon},$
and
\begin{equation}
\left\{
\begin{array}
[c]{l}%
\Delta\Phi=0,\text{ in }\Omega^{\varepsilon},\\
\Phi_{\nu}+H\Phi=0,\text{ on }\partial\Omega^{\varepsilon}.
\end{array}
\right.  \label{stable}%
\end{equation}
Here $\nu$ is the outward normal to $\partial\Omega^{\varepsilon}$ and $H$ is
the mean curvature of $\partial\Omega^{\varepsilon}.$
\end{theorem}

By this theorem, there are solutions whose nodal set is close to
$S_{\varepsilon}^{+}$ \textit{for }$\varepsilon$\textit{ small}. It is well
known that the family of minimal surfaces $S_{\delta}^{+},\delta\in
\mathbb{R},$ are all area minimizing. Therefore, it is natural to ask that
whether the solutions $u_{\varepsilon}$ are also minimizers of the energy
functional $J.$ We believe this is true, but here in this paper we shall only
give the following.

\begin{theorem}
\label{mini}There exists a nontrivial solution (not one dimensional) $U$ to
the free boundary problem $\left(  \ref{FB}\right)  $ in $\mathbb{R}^{8}$
which is also energy minimizing.
\end{theorem}

With additional efforts, one can actually prove that for each $S_{\delta}%
^{+},$ there exists an energy minimizer whose nodal set is asymptotic to
$S_{\delta}^{+}$ at infinity. We will not pursue this in this paper. One can
compare this result with a similar result for the Allen-Cahn
equation \cite{Liu}.

Using the variational method of Jerison-Monneau \cite{JM}, we can construct monotone
solutions in $\mathbb{R}^{9}$ using this minimizer $U.$ This complements the
result of Kamburov \cite{Kam}, where the existence of monotone solutions is
established by sub and super solution method.

\begin{theorem}
\label{mono}There is a family of solutions in $\mathbb{R}^{9}$ to $\left(
\ref{FB}\right)  $ which are monotone in the $x_{9}$ direction.
\end{theorem}

Our second type of minimal surfaces will be the catenoids, which is a family
of classical minimal surfaces with finite total curvature. They are
rotationally symmetric and given explicitly by the equation
\[
x_{1}^{2}+x_{2}^{2}=\frac{1}{\varepsilon^{2}}\cosh^{2}\left(  \varepsilon
x_{3}\right)  .
\]
Here $\varepsilon>0$ is a parameter. In higher dimensions, we have analogous
codimension one minimal submanifold which we call higher dimensional
catenoids. To be more precise, let $\left(  x_{1},...,x_{n}\right)  $ be the
coordinate in $\mathbb{R}^{n}$ ($n>3$). Let $\phi$ be the solution of%
\[
\left\{
\begin{array}
[c]{l}%
\frac{\omega^{\prime\prime}}{1+\omega^{\prime2}}-\frac{n-2}{\omega}=0,\\
\omega\left(  0\right)  =1,\omega^{\prime}\left(  0\right)  =0.
\end{array}
\right.
\]
Then the surface $\mathcal{C}_{1}$ in $\mathbb{R}^{n}$ given by
\[
r:=\varepsilon\sqrt{x_{1}^{2}+...+x_{n-1}^{2}}=\omega\left(  x_{n}\right)
\]
is a minimal surface, called catenoid. We can also write it as
\[
x_{n}=\bar{\omega}\left(  r\right)  ,r\in\lbrack r_{0},+\infty).
\]
Then there are constants $c_{n},c_{n}^{\prime}$ such that
\[
x_{n}\sim c_{n}-c_{n}^{\prime}r^{3-n}.
\]
Actually a homothety of $\mathcal{C}_{1}$ is also a minimal surface, which we
denoted by $\mathcal{C}_{\varepsilon},$ which is then described by
\[
x_{n}=\bar{\omega}_{\varepsilon}\left(  r\right)  :=\frac{1}{\varepsilon}%
\bar{\omega}\left(  \varepsilon r\right)  .
\]
We refer to \cite{Tam} for more detailed properties on catenoids, including
their Morse index. Here we are interested in $\mathcal{C}_{\varepsilon}$ with
$\varepsilon$ small. In this case, the catenoid has a large waist.

\begin{theorem}
\label{catenoid}For $\varepsilon$ small enough, there exists a rationally
symmetric domain $\Omega^{\varepsilon}$ close to radius one tubular
neighbourhood of $\mathcal{C}_{\varepsilon}$ and a solution $u_{\varepsilon}$
to the free boundary problem $\left(  \ref{FB}\right)  .$
\end{theorem}

Now let us explain the main ideas of the proof.  The proofs of Theorems \ref{simons} and \ref{catenoid} are based on the infinite dimensional gluing methods developed in \cite{Manuel, Manuel2}.   In \cite{Oscar,Manuel}, entire
solutions for the Allen-Cahn equation have been constructed. The zero level
sets of the solutions lie close to certain nondegenerate minimal surfaces. To
construct these solutions, they used the method of infinite dimensional
Lyapunov-Schmidt reduction. More recently, in $\cite{Manuel2},$ an
over-determined problem was investigated using similar method. Here we develop new gluing methods for (\ref{FB}).
There are two main difficulties in performing gluing  methods for (\ref{FB}). The first one is that the one-dimensional solution, which is given by
\begin{equation}
 u_0 (x_1)= \left\{\begin{array}{l}
 -1, x_1 \leq -1;\\
 x_1, -1<x_1 <1; \\
 1, \ x_1\geq 1,
\end{array}
\right.
\end{equation}
is only continuous and is not differentiable. This means that one can not linearize the problem around this one dimensional profile. This is quite different from \cite{Oscar, Manuel, Manuel2}.  The second difficulty is that this is an over-determined problem and we have to adjust two interfaces.

\medskip

  To solve the problem $\left(  \ref{FB}\right)  ,$ we introduce a
pair of unknown functions $\left(  h_{1},h_{2}\right)  $ on a rescaled minimal
surface. Using these two functions, we define a perturbed domain $\Omega_{h}$
which will be very close to the radius one tubular neighbourhood
$\mathcal{N}_{1}$ of the minimal surface$.$ The functions $h_{1}$ and $h_{2}$
measures the deviation of $\Omega_{h}$ to $\mathcal{N}_{1}.$ Next, we define
suitable approximate solutions for $\left(  \ref{FB}\right)  $ on $\Omega_{h}%
$. We analyze in detail the differences between this approximate solution and
the harmonic function in $\Omega$ with Dirichlet boundary condition. In the
last step, we use fixed point argument to show that one can find functions
$h_{1}$ and $h_{2}$ such that our problem is solvable and we can get a
solution $u.$ In this step, we show that to match the required Neumann
boundary condition, we need to analyze the solvability and a priori estimate
of a system of equations for the function $h_{1},h_{2}$. (See (\ref{fg1}).) It turns out that  one of them
reduces to the analyze of the Jacobi operator on the minimal surface
\begin{equation}
 \Delta_M h+|A|^2 h= f
 \end{equation}
but the other problem is of fractional differential operator
\begin{equation}
 (-\Delta_M +1)^{\frac{1}{2}} h= f.
 \end{equation}

\medskip

 We
remark that the family of solutions constructed from the Simons' cone are
ordered and hence stable, while the solutions arising from catenoids are
unstable.

\medskip

To prove Theorems \ref{mini} and \ref{mono}, we first extend the construction of Jerison-Monneau \cite{JM} and follow the variational approach in \cite{Liu} to construct minimizers in
$\mathbb{R}^{8}$ and monotone solutions in $\mathbb{R}^{9}.$  The main difficulty is the regularity of the solutions. To this end, we use axial symmetry of the solutions and  also  make use of classical regularity result of Weiss \cite{Weiss2, Weiss3} as well as recent regularity results of
Jerison-Savin \cite{Jerison1}.

\medskip

\textit{Acknowledgement.} The research of J. Wei is partially supported by
NSERC of Canada. Part of the paper was finished while Y. Liu was visiting the
University of British Columbia in 2016. He appreciates the institution for its
hospitality and financial support. K. Wang is supported by \textquotedblleft
the Fundamental Research Funds for the Central Universities". Y. Liu is
partially supported by the Fundamental Research Funds for the Central
Universities 13MS39.

\section{Solutions from Simons' cone}

\subsection{Preliminary on Simons' cone and the associated foliation}

Let us first of all recall some basic facts about the geometry of the Simons'
cone. Throughout the paper we shall use $S^{k}\left(  \rho\right)  $ to denote
the radius $\rho$ sphere in $\mathbb{R}^{k+1}.$ In the manifold $S^{7}\left(
1\right)  ,$ we shall consider the codimension one submanifold
\[
\Lambda:=S^{3}\left(  \rho\right)  \times S^{3}\left(  \rho\right)  ,
\]
where
\[
\rho=\sqrt{\frac{1}{2}}.
\]
The induced metric on $\Lambda$ is given by $g^{\ast}:=\rho^{2}g_{1}+\rho
^{2}g_{2},$ where $g_{1},g_{2}$ are the metric on the two copies $S^{3}\left(
1\right)  .$ The Simons cone is defined to be
\[
S:=\left\{  rX\in\mathbb{R}^{8}:r\in\left(  0,+\infty\right)  ,X\in
\Lambda\right\}  .
\]
One can verify that this is a minimal hypersurface in $\mathbb{R}^{8}.$ The
induced metric tensor on $S$ is then given by
\[
dr^{2}+r^{2}g^{\ast}.
\]
For a codimension one submanifold $M$ in $\mathbb{R}^{n},$ with the induced
metric, we shall use $J_{M}$ to denote its Jacobi operator, which explicitly
has the form
\[
J_{M}=\Delta_{M}+\left\vert A\right\vert ^{2},
\]
where $\left\vert A\right\vert ^{2}=\Sigma_{i=1}^{n-1}k_{i}^{2}$ is the
squared norm of the second fundamental form of $M,$ with $k_{i}$ being the
principle curvatures of $M.$ The Jacobi operator about $S$ is then given by
\[
J_{S}=\Delta_{S}+\left\vert A\right\vert ^{2}=\partial_{r}^{2}+\frac{6}%
{r}\partial_{r}+\frac{\Delta_{g^{\ast}}+6}{r^{2}}.
\]

The set $\mathbb{R}^{8}\backslash S$ has two components$.$ Each component is
foliated by a family of smooth minimal hypersurfaces $S_{\varepsilon}^{\pm}$
which are asymptotic to $S$ at infinity. We can choose $S_{1}$ to be the
surface having the form%
\[
S_{1}\backslash B_{r_{0}}=\left\{  X+\eta_{0}\left(  X\right)  \nu,X\in
S\right\}  ,
\]
where $\nu$ is a choice of the unit normal at $S,$ and $\eta_{0}\left(
X\right)  =\left\vert X\right\vert ^{-2}+o\left(  \left\vert X\right\vert
^{-2}\right)  .$ Then $S_{\varepsilon}=\varepsilon^{-1}S_{1}.$

Let $x=\sqrt{x_{1}^{2}+...+x_{4}^{2}},$ $y=\allowbreak\sqrt{x_{5}%
^{2}+...+x_{8}^{2}}.$ We can write the standard metric on $\mathbb{R}^{8}$ in
the polar coordinate as
\[
dx^{2}+x^{2}d\theta^{2}+dy^{2}+y^{2}d\bar{\theta}^{2},
\]
where $d\theta^{2}$ and $d\bar{\theta}^{2}$ represents the metric tensor on
the unit three-dimensional sphere $S^{3}\left(  1\right)  .$ Suppose in the
$\left(  x,y\right)  $ coordinate $S_{\delta}$ is described by $y=\varphi
_{\delta}\left(  x\right)  $ for a monotone function $\varphi_{\delta},$ then
the metric tensor on $S_{\delta}$ is
\[
\left[  1+\varphi_{\delta}^{\prime}{}^{2}\left(  x\right)  \right]
dx^{2}+\varphi_{\delta}^{2}\left(  x\right)  d\bar{\theta}^{2}+x^{2}%
d\theta^{2}.
\]
Let us introduce the arc length variable $l$ by the formula%
\[
l=\int_{0}^{x}\sqrt{1+\varphi_{\delta}^{\prime2}\left(  t\right)  }dt.
\]
Then the metric $\mathtt{g}_{\delta}$ on $S_{\delta}$ also read as
\[
dl^{2}+\varphi_{\delta}^{2}\left(  x\right)  d\theta^{2}+x^{2}d\bar{\theta
}^{2}.
\]
Note that $\det\mathtt{g}_{\delta}=\varphi_{\delta}^{6}\left(  x\right)
x^{6}.$ Let $\eta$ be a function on $S_{\delta}$ which is invariant under the
action of the group $O\left(  4\right)  \times O\left(  4\right)  .$ The
Laplacian operator on $S_{\delta}$ acting on function $\eta$ has the form%
\begin{align}
\Delta_{S_{\delta}}\eta &  =\frac{1}{\sqrt{\det\mathtt{g}_{\delta}}}%
\partial_{i}\left(  \sqrt{\det\mathtt{g}_{\delta}}\mathtt{g}_{\delta}%
^{i,j}\partial_{j}\eta\right) \nonumber\\
&  =\frac{d^{2}\eta}{dl^{2}}+\frac{\frac{d\left[  \varphi_{\delta}^{3}\left(
x\right)  x^{3}\right]  }{dl}}{\varphi_{\delta}^{3}\left(  x\right)  x^{3}%
}\frac{d\eta}{dl}\nonumber\\
&  =\frac{d^{2}\eta}{dl^{2}}+\left(  \frac{3}{x}+\frac{3\varphi_{\delta
}^{\prime}}{\varphi_{\delta}}\right)  \frac{dx}{dl}\frac{d\eta}{dl}.
\label{L1}%
\end{align}

\subsection{Analysis of the approximate solutions}

We will construct solutions based on the minimal hypersurfaces $S_{\varepsilon
}$ where $\varepsilon>0$ is sufficiently small. Let us choose a unit normal
$\nu$ for the codimension one manifold $S_{\varepsilon}.$ Let $h_{1},h_{-1}\in
C_{loc}^{2,\alpha}\left(  S_{\varepsilon}\right)  ,$ small in certain sense.
For each function $\eta$ defined on $S_{\varepsilon},$ we set
\[
\Gamma_{\eta}:=\left\{  X+\eta\left(  X\right)  \nu\left(  X\right)  :X\in
S_{\varepsilon}\right\}  .
\]
Although $\Gamma_{\eta}$ depends also on $\varepsilon,$ we will not make this
dependence explicit in the notation. We establish a Fermi coordinate in a
tubular neighbourhood of $S_{\varepsilon}.$ By $s$ we denote the signed
distance of a point to $S_{\varepsilon}.$ Slightly abusing the notation,
define
\[
\Gamma_{s}:=\left\{  X+s\nu\left(  X\right)  :X\in S_{\varepsilon}\right\}  .
\]
Note that for $\varepsilon$ small, this is well defined and $\Gamma_{s}$ is
smooth, for all $|s|<1$.

Let us consider the region $\Omega$ trapped between the surfaces
$\Gamma_{-1+h_{-1}}$ and $\Gamma_{1+h_{1}}.$ For each pair of functions
$h=\left(  h_{-1},h_{1}\right)  ,$ we shall define an approximate solution
$w_{h}$ in $\Omega:$%
\[
w_{h}\left(  s,l\right)  =\frac{s-g\left(  l\right)  }{1+f\left(  l\right)
},
\]
where
\begin{align*}
f  &  =\frac{h_{1}-h_{-1}}{2},\\
g  &  =\frac{h_{1}+h_{-1}}{2}.
\end{align*}
Note that in the current situation, the range of $l$ is $[0,+\infty).$ With
this definition, $w_{h}$ satisfies the boundary condition:%

\[
w_{h}=\left\{
\begin{array}
[c]{l}%
-1,\text{ on }\Gamma_{-1+h_{-1}},\\
1,\text{ on }\Gamma_{1+h_{1}}.
\end{array}
\right.
\]
It will be convenient for us to introduce a new variable
\[
t=\frac{s-g\left(  l\right)  }{1+f\left(  l\right)  }.
\]
Then the domain $\Omega_{h}$ can be parameterized by $\left(  l,t\right)  $
with $t\in\left[  -1,1\right]  .$

Let us use $H_{M}$ to denote the mean curvature of a codimension one
submanifold $M.$ The formula of Laplacian operator in the Fermi coordinate
(see \cite{DW}) tells us that
\begin{align*}
\Delta w_{h}\left(  s,l\right)   &  =\Delta_{\Gamma_{s}}w_{h}+\partial_{s}%
^{2}w_{h}-H_{\Gamma_{s}}\partial_{s}w_{h}\\
&  =\Delta_{\Gamma_{s}}w_{h}-\frac{H_{\Gamma_{s}}}{1+f}.
\end{align*}
We need to understand the main order of these terms.

\begin{lemma}
\label{l1}We have
\[
\Delta_{\Gamma_{0}}w_{h}=-\Delta_{\Gamma_{0}}g-t\Delta_{\Gamma_{0}}f+E_{1},
\]
where
\[
E_{1}=-tf\Delta_{\Gamma_{0}}f+\Delta_{\Gamma_{0}}\left(  fg\right)
-g\Delta_{\Gamma_{0}}f+\Delta_{\Gamma_{0}}\left[  \left(  s-g\right)
\frac{f^{2}}{1+f}\right]  .
\]

\end{lemma}

\begin{remark}
$E_{1}$ can be regarded as a perturbation term.
\end{remark}

\begin{proof}
Having in mind that $f,g$ are small, we write
\begin{align*}
w_{h}  &  =\frac{s-g\left(  l\right)  }{1+f\left(  l\right)  }=\left(
s-g\right)  \left(  1-f+\frac{f^{2}}{1+f}\right) \\
&  =s-g-sf+gf+\left(  s-g\right)  \frac{f^{2}}{1+f}\text{. }%
\end{align*}
We then compute
\[
\Delta_{\Gamma_{0}}w_{h}=-\Delta_{\Gamma_{0}}g-s\Delta_{\Gamma_{0}}%
f+\Delta_{\Gamma_{0}}\left(  fg\right)  +\Delta_{\Gamma_{0}}\left[  \left(
s-g\right)  \frac{f^{2}}{1+f}\right]  .
\]
Inserting the relation $s=t\left(  1+f\right)  +g$ into the left hand side, we
get%
\begin{align*}
\Delta_{\Gamma_{0}}w_{h}  &  =-\Delta_{\Gamma_{0}}g-\left[  t\left(
1+f\right)  +g\right]  \Delta_{\Gamma_{0}}f+\Delta_{\Gamma_{0}}\left(
fg\right)  +\Delta_{\Gamma_{0}}\left[  \left(  s-g\right)  \frac{f^{2}}%
{1+f}\right] \\
&  =-\Delta_{\Gamma_{0}}g-t\Delta_{\Gamma_{0}}f-tf\Delta_{\Gamma_{0}}%
f+\Delta_{\Gamma_{0}}\left(  fg\right)  -g\Delta_{\Gamma_{0}}f+\Delta
_{\Gamma_{0}}\left[  \left(  s-g\right)  \frac{f^{2}}{1+f}\right]  .
\end{align*}
This finishes the proof.
\end{proof}

Let us use $k_{i},i=1,...,6$ to denote the principle curvatures of
$S_{\varepsilon}.$

\begin{lemma}
\label{l2}We have the following formula:%
\[
\frac{H_{\Gamma_{s}}}{1+f}=t\left\vert A\right\vert ^{2}+g\left\vert
A\right\vert ^{2}+E_{2},
\]
where
\[
E_{2}=\frac{1}{1+f}\sum\limits_{i=1}^{6}\frac{s^{2}k_{i}^{3}}{1-sk_{i}}%
-\frac{fg\left\vert A\right\vert ^{2}}{1+f}.
\]

\end{lemma}

\begin{proof}
By a well known formula (see \cite{DW}),
\[
H_{\Gamma_{s}}=\sum\limits_{i=1}^{6}\frac{k_{i}}{1-sk_{i}}=\sum\limits_{i=1}%
^{6}k_{i}+\sum\limits_{i=1}^{6}sk_{i}^{2}+\sum\limits_{i=1}^{6}\frac
{s^{2}k_{i}^{3}}{1-sk_{i}}.
\]
Recall that $\sum\limits_{i=1}^{6}k_{i}=H_{\Gamma_{0}}=0.$ Hence
\begin{align*}
\frac{H_{\Gamma_{s}}}{1+f}  &  =\frac{\left\vert A\right\vert ^{2}}%
{1+f}\left[  \left(  1+f\right)  t+g\right]  +\frac{1}{1+f}\sum\limits_{i=1}%
^{6}\frac{s^{2}k_{i}^{3}}{1-sk_{i}}\\
&  =t\left\vert A\right\vert ^{2}+g\left\vert A\right\vert ^{2}-\frac
{fg\left\vert A\right\vert ^{2}}{1+f}+\frac{1}{1+f}\sum\limits_{i=1}^{6}%
\frac{s^{2}k_{i}^{3}}{1-sk_{i}}.
\end{align*}
The proof is thus completed.
\end{proof}

We seek a solution $u$ to the free boundary problem $\left(  \ref{FB}\right)
$ in the form $u=w_{h}+\phi.$ Here we require $\phi=0$ on $\partial\Omega
_{h}.$ Let us now analyze the boundary condition $\left\vert \nabla
u\right\vert =1$ on $\partial\Omega_{h}.$ Suppose in the $\left(
l,\theta,\bar{\theta},s\right)  $ coordinate the metric tensor $\mathfrak{g}$
in a tubular neighbourhood of $S_{\varepsilon}$ has matrix with entries
$\mathfrak{g}_{i,j}$ and its inverse matrix has entries $\mathfrak{g}^{i,j}.$
Since we are working in the Fermi coordinate, the entries in the last column
and row are all zero, except the rightmost entry on the last row. We omit the
subscript $h$ in $w_{h}$ and write it as $w.$

\begin{lemma}
\label{l3}The condition $\left\vert \nabla u\right\vert =1$ on $\Gamma
_{i+h_{i}}$ is equivalent to%
\[
\partial_{t}\phi-f=E_{3,i}.
\]
Here for $i=-1,1,$ $E_{3,i}$ is defined on $\Gamma_{i+h_{i}}$ to be
\[
-\frac{1}{2}\left(  1+\mathfrak{g}^{1,1}h_{i}^{\prime2}\right)  \left(
\partial_{t}\phi\right)  ^{2}+\frac{\mathfrak{g}^{1,1}h_{i}^{\prime}}%
{1+f}\partial_{t}\phi+\frac{1}{2}f^{2}-\frac{1}{2}\mathfrak{g}^{1,1}\left(
g^{\prime}+tf^{^{\prime}}\right)  ^{2}.\text{ }%
\]

\end{lemma}

\begin{proof}
We compute the norm of the gradient in the $\left(  s,l\right)  $ coordinate
and get the following equation to be satisfied on the boundary $\partial
\Omega_{h}:$
\begin{equation}
\left\vert \nabla\left(  w+\phi\right)  \right\vert ^{2}=\left(  \partial
_{s}w+\partial_{s}\phi\right)  ^{2}+\mathfrak{g}^{1,1}\left(  \partial
_{l}w+\partial_{l}\phi\right)  ^{2}=1. \label{G1}%
\end{equation}
Direct computation yields
\[
\partial_{s}w=\frac{1}{1+f},
\]
and
\[
\partial_{l}w=\frac{-g^{\prime}}{1+f}-\frac{\left(  s-g\right)  f^{\prime}%
}{\left(  1+f\right)  ^{2}}.
\]
On the other hand, differentiating the identity $\phi\left(  -1+h_{1}%
,l\right)  =0$ with respect to $l,$ we obtain%
\[
\partial_{l}\phi=-\partial_{s}\phi h_{1}^{\prime}\text{ on }\Gamma_{-1+h_{-1}%
}.
\]

On $\Gamma_{-1+h_{-1}},$ the right hand side of $\left(  \ref{G1}\right)  $ is
equivalent to
\begin{equation}
\left(  1+\mathfrak{g}^{1,1}h_{1}^{\prime2}\right)  \left(  \partial_{s}%
\phi\right)  ^{2}+\left(  2\partial_{s}w-2\mathfrak{g}^{1,1}h_{1}^{\prime
}\right)  \partial_{s}\phi+\left(  \partial_{s}w\right)  ^{2}+\mathfrak{g}%
^{1,1}\left(  \partial_{l}w\right)  ^{2}=1. \label{G2}%
\end{equation}
Inserting the equation
\[
\partial_{s}\phi=\frac{\partial_{t}\phi}{1+f}%
\]
into $\left(  \ref{G2}\right)  ,$ we get
\[
\left(  1+\mathfrak{g}^{1,1}h_{1}^{\prime2}\right)  \left(  \partial_{t}%
\phi\right)  ^{2}+\left(  2-2\frac{\mathfrak{g}^{1,1}h_{1}^{\prime}}%
{1+f}\right)  \partial_{t}\phi-2f-f^{2}+\mathfrak{g}^{1,1}\left(  g^{\prime
}+\frac{\left(  s-g\right)  f^{\prime}}{1+f}\right)  ^{2}=0.
\]
This completes the proof.
\end{proof}

The function $\phi$ should also satisfy
\[
\Delta\phi=-\Delta w=J_{\Gamma_{0}}g+\left(  \Delta_{\Gamma_{0}}f+\left\vert
A\right\vert ^{2}\right)  t-E_{1}+E_{2}+\Delta_{\Gamma_{0}}w-\Delta
_{\Gamma_{s}}w,\text{ in }\Omega_{h}.
\]
Here we recall that by $J_{\Gamma_{0}}$ we denote the Jacobi operator of
$\Gamma_{0}.$ Therefore, we are lead to solve the following nonlinear problem
for the unknown functions $\left(  f,g,\phi\right)  .$%

\begin{equation}
\left\{
\begin{array}
[c]{l}%
\Delta\phi=J_{\Gamma_{0}}g+\left(  \Delta_{\Gamma_{0}}f+\left\vert
A\right\vert ^{2}\right)  t-E_{1}+E_{2}+\Delta_{\Gamma_{0}}w-\Delta
_{\Gamma_{s}}w,\text{ in }\Omega_{h},\\
\phi=0\text{ and }\partial_{t}\phi-f=E_{3,i},\text{on }\partial\Omega_{h}.
\end{array}
\right.  \label{fi1}%
\end{equation}

\begin{lemma}
\label{Lap}We have the following estimate for the Laplacian operator acting on
functions depending on $s$ and $l:$%
\[
\Delta_{\Gamma_{0}}\eta- \partial_{l}^{2}\eta-\frac{3}{l}\partial_{l}%
\eta=O\left(  \frac{\varepsilon}{1+\varepsilon l}\right)  \partial_{l}\eta,
\]
and%
\begin{equation}
\Delta_{\Gamma_{s}}\eta-\Delta_{\Gamma_{0}}\eta=O\left(  \frac{\varepsilon
^{2}}{\left(  1+\varepsilon l\right)  ^{2}}\right)  \partial_{l}\eta+O\left(
\frac{\varepsilon}{1+\varepsilon l}\right)  \partial_{l}^{2}\eta. \label{s0}%
\end{equation}

\end{lemma}

\begin{proof}
By $\left(  \ref{L1}\right)  ,$ we have%
\[
\Delta_{\Gamma_{0}}\eta-\frac{d^{2}\eta}{dl^{2}}-\frac{3}{l}\frac{d\eta}%
{dl}=\left[  \left(  \frac{3}{x}+\frac{3\varphi_{\varepsilon}^{\prime}%
}{\varphi_{\varepsilon}}\right)  \frac{dx}{dl}-\frac{3}{l}\right]  \frac
{d\eta}{dl}.
\]
We compute
\begin{align*}
\left(  \frac{3}{x}+\frac{3\varphi_{\varepsilon}^{\prime}}{\varphi
_{\varepsilon}}\right)  \frac{dx}{dl}-\frac{3}{l}  &  =\frac{1}{\sqrt
{1+\varphi_{\varepsilon}^{\prime2}}}\left(  \frac{3}{x}+\frac{3\varphi
_{\varepsilon}^{\prime}}{\varphi_{\varepsilon}}\right)  -\frac{3}{l}\\
&  =\frac{1}{\sqrt{1+\left(  \varphi_{1}^{\prime}\left(  \varepsilon x\right)
\right)  ^{2}}}\left(  \frac{3}{x}+\frac{3\varepsilon\varphi_{1}^{\prime
}\left(  \varepsilon x\right)  }{\varphi_{1}\left(  \varepsilon x\right)
}\right)  -\frac{3}{l}\\
&  =3\frac{l-x\sqrt{1+\left(  \varphi_{1}^{\prime}\left(  \varepsilon
x\right)  \right)  ^{2}}}{lx\sqrt{1+\left(  \varphi_{1}^{\prime}\left(
\varepsilon x\right)  \right)  ^{2}}}+\varepsilon O\left(  \frac
{1}{1+\varepsilon l}\right) \\
&  =O\left(  \frac{\varepsilon}{1+\varepsilon l}\right)  .
\end{align*}

Next we prove $\left(  \ref{s0}\right)  .$ Let us denote by $\mathfrak{g}_{s}$
the metric tensor of $\Gamma_{s}.$ Explicitly, $\mathfrak{g}_{s}\left(
l,\theta,\bar{\theta}\right)  =\mathfrak{g}\left(  l,\theta,\bar{\theta
},s\right)  .$ From the calculation in \cite{DW}, we know that
\[
\sqrt{\det\mathfrak{g}_{s}}=\sqrt{\det\mathfrak{g}_{0}}\prod\limits_{i=1}%
^{6}\left(  1-k_{i}s\right)  ,
\]
where $k_{i}$ are the principle curvatures of $\Gamma_{0}=S_{\varepsilon}.$
Hence, for a function $\eta$ depending on $s$ and $l,$
\begin{align*}
\Delta_{\Gamma_{s}}\eta &  =\frac{1}{\sqrt{\det\mathfrak{g}_{s}}}\partial
_{i}\left(  \sqrt{\det\mathfrak{g}_{s}}\mathfrak{g}_{s}^{i,j}\partial_{j}%
\eta\right) \\
&  =\partial_{l}\left(  \ln\left(  \sqrt{\det\mathfrak{g}_{0}}\prod
\limits_{i=1}^{6}\left(  1-k_{i}s\right)  \right)  \right)  \mathfrak{g}%
_{s}^{1,1}\partial_{l}\eta+\partial_{l}\left(  \mathfrak{g}_{s}^{1,1}%
\partial_{l}\eta\right)  .
\end{align*}
Consequently,
\begin{align*}
\Delta_{\Gamma_{s}}\eta-\Delta_{\Gamma_{0}}\eta &  =\partial_{l}\left(
\ln\left(  \prod\limits_{i=1}^{6}\left(  1-k_{i}s\right)  \right)  \right)
\mathfrak{g}_{s}^{1,1}\partial_{l}\eta\\
&  +\partial_{l}\left(  \ln\sqrt{\det\mathfrak{g}_{0}} \right)  \left(
\mathfrak{g}_{s}^{1,1}-\mathfrak{g}_{0}^{1,1}\right)  \partial_{l}\eta\\
&  +\partial_{l}\left(  \left(  \mathfrak{g}_{s}^{1,1}-\mathfrak{g}_{0}%
^{1,1}\right)  \partial_{l}\eta\right)  .
\end{align*}
Then the desired estimate follows from the fact that
\[
\left\vert \frac{dk_{i}}{dl}\right\vert \leq C\frac{\varepsilon^{2}}{\left(
1+\varepsilon l\right)  ^{2}}.
\]

\end{proof}

By the previous computations, the term $-E_{1}+E_{2}+\Delta_{\Gamma_{0}%
}w-\Delta_{\Gamma_{s}}w$ will be small and can be regarded as perturbation terms.

To get a solution $\left(  f,g,\phi\right)  $ for the original problem, let us
introduce the functional framework to work with. Let $\alpha\in\left(
0,1\right)  $ be a fixed constant. Note that the functions $f$ and $g$ are
both defined on the minimal surface $S_{\varepsilon}.$ However, we shall work
both in functional spaces defined on $S_{\varepsilon}$ and $S_{1}.$ Hence we
introduce the following

\begin{definition}
For $\mu=0,1,2,$ $\beta\geq0,\delta>0,$ the space $\mathcal{B}_{\beta
,\mu;\delta}$ consists of those functions $\eta$ defined on $S_{\delta}$ such
that
\[
\left\Vert \eta\right\Vert _{\beta,\mu;\delta}:=\sup_{\left\vert z\right\vert
=l}\left[  \left(  1+\delta l\right)  ^{\beta}\left\Vert \eta\right\Vert
_{C^{\mu,\alpha}\left(  S_{\delta}\cap B_{1}\left(  z\right)  \right)
}\right]  <+\infty.
\]

\end{definition}

\begin{definition}
The space $\mathcal{\bar{B}}_{\beta,2;\delta}$ consists of those functions
$\eta$ defined on $S_{\delta}$ such that
\begin{align*}
\left\Vert \eta\right\Vert _{\beta,2;\delta,\symbol{94}}  &  :=\sup
_{\left\vert z\right\vert =l}\left[  \left(  1+\delta l\right)  ^{\beta
}\left\Vert \eta\right\Vert _{C^{0,\alpha}\left(  S_{\delta}\cap B_{1}\left(
z\right)  \right)  }+\left(  1+\delta l\right)  ^{\beta+1}\left\Vert
\eta^{\prime}\right\Vert _{C^{0,\alpha}\left(  S_{\delta}\cap B_{1}\left(
z\right)  \right)  }\right] \\
&  +\sup_{\left\vert z\right\vert =l}\left[  \left(  1+\delta l\right)
^{\beta+2}\left\Vert \eta^{\prime\prime}\right\Vert _{C^{0,\alpha}\left(
S_{\delta}\cap B_{1}\left(  z\right)  \right)  }\right]  <+\infty.
\end{align*}

\end{definition}

With the above definition, we shall assume a priori $f\in\mathcal{B}%
_{2,2;\varepsilon}.$ We also assume the rescaled function $\bar{g}\left(
\cdot\right)  =g\left(  \frac{\cdot}{\varepsilon}\right)  \in\mathcal{\bar{B}%
}_{\beta_{0},2;1},$ where $\beta_{0}>2$ is a fixed constant with $\beta_{0}-2$
small. On the other hand, the function $\phi$ is defined on $\Omega_{h},$
which depends on $f$ and $g.$ This turns out to be not very convenient for our
later purpose. Hence slightly abusing the notation, we also regard $\phi$ as
the restriction of a function $\mathcal{T}\left(  \phi\right)  $ on
$\Xi:=\left[  -1,1\right]  \times\mathbb{[}0,+\infty),$ where $\mathcal{T}%
\left(  \phi\right)  $ is a function of $t$ and $l$ defined for $\left(
t,l\right)  \in\bar{\Xi}:=\left[  -1,1\right]  \times\mathbb{R}$, even in the
variable $l$.

\begin{definition}
For $\mu=0,1,2,$ $\beta\geq0$, the space $\mathcal{B}_{\beta,\mu;\ast}$
consists of those functions $\phi$ such that
\[
\left\Vert \phi\right\Vert _{\beta,\mu;\ast}:=\sup_{l\in\mathbb{R};z\in
\bar{\Xi},\left\vert z\right\vert =\left\vert l\right\vert }\left[  \left(
1+\varepsilon\left\vert l\right\vert \right)  ^{\beta}\left\Vert
\mathcal{T}\left(  \phi\right)  \right\Vert _{C^{\mu,\alpha}\left(  \bar{\Xi
}\cap B_{1}\left(  z\right)  \right)  }\right]  <+\infty.
\]

\end{definition}

We shall assume $\phi\in\mathcal{B}_{2,2;\ast}.$ The following invertibility
property of the Jacobi operator on $S_{1}$ will play an important role in our analysis.

\begin{lemma}
\label{lem inverse Jacobian op} For each function $\xi\in\mathcal{B}%
_{\beta_{0}+2,0;1},$ there is a solution $\eta\in\mathcal{\bar{B}}_{\beta
_{0},2;1}$ such that%
\[
J_{S_{1}}\left(  \eta\right)  =\xi,
\]
Moreover, it satisfies%
\[
\left\Vert \eta\right\Vert _{\beta_{0},2;1,\symbol{94}}\leq C\left\Vert
\xi\right\Vert _{\beta_{0}+2,0;1}.
\]

\end{lemma}

\begin{proof}
The proof of this lemma goes in a similar fashion as that of \cite{PW}, we
omit the details.
\end{proof}

We would like to solve the nonlinear problem $\left(  \ref{fi1}\right)  $
using fixed point arguments.

\begin{lemma}
\label{estimate1}For each $\eta\in\mathcal{B}_{\beta,0;\ast},$ there exists a
unique solution $\phi\in\mathcal{B}_{\beta,2;\ast},$ to the problem%
\begin{equation}
\left\{
\begin{array}
[c]{l}%
\partial_{t}^{2}\phi+\partial_{l}^{2}\phi+\frac{3}{l}\partial_{l}\phi
=\eta,\text{ in }\Omega_{h},\\
\phi=0\text{ on }\partial\Omega_{h},
\end{array}
\right.  \label{a}%
\end{equation}
with $\left\Vert \phi\right\Vert _{\beta,2;\ast}\leq C\left\Vert
\eta\right\Vert _{\beta,0;\ast}.$ This solution will be denoted by
$L_{1}\left(  \eta\right)  .$
\end{lemma}

\begin{remark}
In terms of the $\left(  t,l\right)  $ coordinate, the first equation in
$\left(  \ref{a}\right)  $ actually should be considered in the region
$\left(  t,l\right)  \in\left[  -1,1\right]  \times\lbrack0,+\infty).$
However, for the sake of notational simplicity, we just write it as in
$\Omega_{h}.$ Similarly, we use the notation $\partial\Omega_{h}$ in the
second equation of $\left(  \ref{a}\right)  .$
\end{remark}

The proof of Lemma \ref{estimate1} follows from standard arguments.

Next, given two functions $\gamma_{1}$ and $\gamma_{-1}$ defined on
$\mathcal{S}_{\varepsilon}$, we consider
\begin{equation}
\left\{
\begin{array}
[c]{l}%
\partial_{t}^{2}\phi+\partial_{l}^{2}\phi+\frac{3}{l}\partial_{l}%
\phi=J_{\Gamma_{0}}g+\left(  \Delta_{\Gamma_{0}}f+\left\vert A\right\vert
^{2}\right)  t,\text{ in }\Omega_{h},\\
\phi\left(  \pm1,l\right)  =0,\\
\partial_{t}\phi-f=\gamma_{-1},\text{for }t=-1,\\
\partial_{t}\phi-f=\gamma_{1},\text{ for }t=1.
\end{array}
\right.  \label{linear2}%
\end{equation}
To find the explicit form of the solution $\phi$ of this problem, we need to
introduce some notations. For each fixed $\xi\in\mathbb{R}^{4},$ let us use
$p_{1,\xi}\left(  \cdot\right)  $ to denote the solution of the problem
\[
\left\{
\begin{array}
[c]{l}%
p_{1,\xi}^{\prime\prime}\left(  t\right)  -\left\vert \xi\right\vert
^{2}p_{1,\xi}\left(  t\right)  =1,\\
p_{1,\xi}\left(  -1\right)  =p_{1,\xi}\left(  1\right)  =0.
\end{array}
\right.
\]
We use $p_{2,\xi}\left(  \cdot\right)  $ to denote the solution of
\[
\left\{
\begin{array}
[c]{l}%
p_{2,\xi}^{\prime\prime}\left(  t\right)  -\left\vert \xi\right\vert
^{2}p_{2,\xi}\left(  t\right)  =t,\\
p_{2,\xi}\left(  -1\right)  =p_{2,\xi}\left(  1\right)  =0.
\end{array}
\right.
\]
Note that $p_{1,\xi}$ is even, while $p_{2,\xi}$ is odd. For convenience, we
collect properties of $p_{i,\xi}$ in the following

\begin{lemma}
\label{pi}Explicitly,
\begin{align*}
p_{1,\xi}\left(  t\right)   &  =\frac{\cosh\left(  \left\vert \xi\right\vert
t\right)  }{\left\vert \xi\right\vert ^{2}\cosh\left\vert \xi\right\vert
}-\frac{1}{\left\vert \xi\right\vert ^{2}},\\
p_{2,\xi}\left(  t\right)   &  =\frac{\sinh\left(  \left\vert \xi\right\vert
t\right)  }{\left\vert \xi\right\vert ^{2}\sinh\left\vert \xi\right\vert
}-\frac{t}{\left\vert \xi\right\vert ^{2}}.
\end{align*}
Moreover,
\[
\frac{1}{p_{1,\xi}^{\prime}\left(  1\right)  }-\left\vert \xi\right\vert
=\frac{\left\vert \xi\right\vert }{\tanh\left\vert \xi\right\vert }-\left\vert
\xi\right\vert =O\left(  e^{-\frac{\left\vert \xi\right\vert }{2}}\right)
,\text{ as }\left\vert \xi\right\vert \rightarrow+\infty,
\]
and
\[
\left\vert \xi\right\vert ^{2}p_{2,\xi}^{\prime}\left(  1\right)
=\frac{\left\vert \xi\right\vert }{\tanh\left\vert \xi\right\vert }-1.
\]

\end{lemma}

\begin{proof}
This follows from direct computation.
\end{proof}

In the following, we shall use the following Fourier type transform%
\[
\hat{\eta}\left(  t,\xi\right)  :=\int_{\mathbb{R}^{4}}e^{-2\pi i\left(
\xi_{1}z_{1}+...+\xi_{4}z_{4}\right)  }\eta\left(  t,l\right)  dz_{1}%
...dz_{4},
\]
where $l=\sqrt{z_{1}^{2}+...+z_{4}^{2}},\xi=\left(  \xi_{1},\xi_{2},\xi
_{3},\xi_{4}\right)  .$ Note that this actually corresponds to the usual
Fourier transform in $\mathbb{R}^{4}.$ We denote by $\left(  \cdot\right)
^{\vee}.$ Define a new function $f_{0}$ by
\[
f_{0}=-\left(  \frac{\left(  \left\vert A\right\vert ^{2}\right)
^{\symbol{94}}}{\left\vert \xi\right\vert ^{2}-\frac{1}{p_{2,\xi}^{\prime
}\left(  1\right)  }}\right)  ^{\vee}.
\]
By the discussion in the next proposition, this definition makes sense.

\begin{proposition}
\label{L2}Suppose $\gamma_{1}-\gamma_{-1}\in\mathcal{B}_{\beta_{0}%
+2,1;\varepsilon},\gamma_{1}+\gamma_{-1}\in\mathcal{B}_{\beta_{0}%
,1;\varepsilon}.$ Then the system $\left(  \ref{linear2}\right)  $ has a
solution $\left(  f,\bar{g}\right)  $ with%
\begin{equation}
\left\Vert f-f_{0}\right\Vert _{\beta_{0},2;\varepsilon}\leq C\left\Vert
\gamma_{1}+\gamma_{-1}\right\Vert _{\beta_{0},1;\varepsilon}, \label{equ1}%
\end{equation}
and
\begin{equation}
\left\Vert \bar{g}\right\Vert _{\beta_{0},2;1,\symbol{94}}\leq C\varepsilon
^{-2}\left\Vert \gamma_{1}-\gamma_{-1}\right\Vert _{\beta_{0}+2,1;\varepsilon
}. \label{equ2}%
\end{equation}
This solution $\left(  f,\bar{g}\right)  $ will be denoted by $L_{2}\left(
\gamma_{-1},\gamma_{1}\right)  .$
\end{proposition}

\begin{proof}
We are lead to the problem%
\begin{equation}
\left\{
\begin{array}
[c]{l}%
\partial_{t}^{2}\hat{\phi}-\left\vert \xi\right\vert ^{2}\hat{\phi}=\left(
J_{\Gamma_{0}}g\right)  ^{\symbol{94}}+\left(  \Delta_{\Gamma_{0}}f+\left\vert
A\right\vert ^{2}\right)  ^{\symbol{94}}t,\text{ }t\in\left[  -1,1\right]  ,\\
\hat{\phi}\left(  -1,\xi\right)  =\hat{\phi}\left(  1,\xi\right)  =0,\\
\partial_{t}\hat{\phi}\left(  -1,\xi\right)  -\hat{f}\left(  \xi\right)
=\hat{\gamma}_{-1}\left(  \xi\right)  ,\\
\partial_{t}\hat{\phi}\left(  1,\xi\right)  -\hat{f}\left(  \xi\right)
=\hat{\gamma}_{1}\left(  \xi\right)  \text{.}%
\end{array}
\right.  \label{linear1}%
\end{equation}
The solution $\hat{\phi}$ of the first equation in $\left(  \ref{linear1}%
\right)  $ can be written in the form
\[
\hat{\phi}\left(  t,\xi\right)  =\left(  J_{\Gamma_{0}}g\right)
^{\symbol{94}}p_{1,\xi}\left(  t\right)  +\left(  \Delta_{\Gamma_{0}%
}f+\left\vert A\right\vert ^{2}\right)  ^{\symbol{94}}p_{2,\xi}\left(
t\right)  .
\]
Therefore, to get a solution for $\left(  \ref{linear1}\right)  $, it suffices
for us to solve the following problem:%
\begin{equation}
\left\{
\begin{array}
[c]{l}%
\left(  J_{\Gamma_{0}}g\right)  ^{\symbol{94}}p_{1,\xi}^{\prime}\left(
-1\right)  +\left(  \Delta_{\Gamma_{0}}f+\left\vert A\right\vert ^{2}\right)
^{\symbol{94}}p_{2,\xi}^{\prime}\left(  -1\right)  -\hat{f}\left(  \xi\right)
=\hat{\gamma}_{-1}\left(  \xi\right)  ,\\
\left(  J_{\Gamma_{0}}g\right)  ^{\symbol{94}}p_{1,\xi}^{\prime}\left(
1\right)  +\left(  \Delta_{\Gamma_{0}}f+\left\vert A\right\vert ^{2}\right)
^{\symbol{94}}p_{2,\xi}^{\prime}\left(  1\right)  -\hat{f}\left(  \xi\right)
=\hat{\gamma}_{1}\left(  \xi\right)  .
\end{array}
\right.  \label{fg}%
\end{equation}
Due to the symmetry of $p_{1,\xi}$ and $p_{2,\xi},$ $\left(  \ref{fg}\right)
$ is equivalent to
\begin{equation}
\left\{
\begin{array}
[c]{l}%
\left(  J_{\Gamma_{0}}g\right)  ^{\symbol{94}}=\ \frac{\hat{\gamma}_{1}\left(
\xi\right)  -\hat{\gamma}_{-1}\left(  \xi\right)  }{2p_{1,\xi}^{\prime}\left(
1\right)  },\\
\left(  \Delta_{\Gamma_{0}}f+\left\vert A\right\vert ^{2}\right)
^{\symbol{94}}=\frac{2\hat{f}\left(  \xi\right)  +\hat{\gamma}_{-1}\left(
\xi\right)  +\hat{\gamma}_{1}\left(  \xi\right)  }{2p_{2,\xi}^{\prime}\left(
1\right)  }.
\end{array}
\right.  \label{fg1}%
\end{equation}
One can perform inverse Fourier transform for the first equation in this
system and then use Lemma \ref{lem inverse Jacobian op} to get a solution $g.$

We proceed to estimate the norm of $\bar{g}\left(  \cdot\right)  =g\left(
\frac{\cdot}{\varepsilon}\right) $. Put $\rho=\gamma_{1}-\gamma_{-1}.$ We
would like to show
\[
\ \left\Vert \left(  \frac{\hat{\rho}\left(  \xi\right)  }{p_{1,\xi}^{\prime
}\left(  1\right)  }\right)  ^{\vee}\right\Vert _{\beta_{0}+2,0;\varepsilon
}\leq C\left\Vert \rho\right\Vert _{\beta_{0}+2,1;\varepsilon}.
\]
Once this is proved, the estimate $\left(  \ref{equ2}\right)  $ follows from
the invertibility property of the Jacobi operator $J_{S_{1}}.$ Observe that
$\frac{1}{p_{1,\xi}^{\prime}\left(  1\right)  }$ is real analytic in
$\left\vert \xi\right\vert .$ By Lemma \ref{pi},
\[
\frac{1}{p_{1,\xi}^{\prime}\left(  1\right)  }=\left\vert \xi\right\vert
+O\left(  e^{-\frac{\left\vert \xi\right\vert }{2}}\right)  ,\text{ as
}\left\vert \xi\right\vert \rightarrow+\infty.
\]
Let us now estimate the inverse Fourier transform of $\left\vert
\xi\right\vert \hat{\rho}\left(  \xi\right)  .$ Using the fact that in
$\mathbb{R}^{4},$ inverse Fourier transform of $\left\vert \xi\right\vert $ is
equal to $c_{0}\left\vert x\right\vert ^{-5}$, where $c_{0}$ is a contant(see
for instances, \cite{G} Theorem 2.4.6, or \cite{Valdinoci}), we get
\[
\left(  \left\vert \xi\right\vert \hat{\rho}\left(  \xi\right)  \right)
^{\vee}\left(  z\right)  =c_{0}\text{P.V.}\int_{\mathbb{R}^{4}}\frac
{\rho\left(  \left\vert z\right\vert \right)  -\rho\left(  \left\vert
y\right\vert \right)  }{\left\vert z-y\right\vert ^{5}}dy.
\]
For $\left\vert z\right\vert $ large, we have%
\begin{align}
\left\vert \int_{\left\vert z-y\right\vert >\frac{\left\vert z\right\vert }%
{2}}\frac{\rho\left(  \left\vert z\right\vert \right)  -\rho\left(  \left\vert
y\right\vert \right)  }{\left\vert z-y\right\vert ^{5}}dy\right\vert  &  \leq
C\left\vert \rho\left(  \left\vert z\right\vert \right)  \right\vert
+\int_{\left\vert z-y\right\vert >\frac{\left\vert z\right\vert }{2}}%
\frac{\left\vert \rho\left(  \left\vert y\right\vert \right)  \right\vert
}{\left\vert z-y\right\vert ^{5}}dy\nonumber\\
&  \leq C\left\vert \rho\left(  \left\vert z\right\vert \right)  \right\vert
+\frac{1}{\left\vert z\right\vert ^{5}}\int_{\left\vert z-y\right\vert
>\frac{\left\vert z\right\vert }{2}}\left\vert \rho\left(  \left\vert
y\right\vert \right)  \right\vert dy\nonumber\\
&  \leq C\left\vert \rho\left(  \left\vert z\right\vert \right)  \right\vert
+\frac{\left\Vert \rho\right\Vert _{\beta_{0}+2,1;\varepsilon}}{1+\varepsilon
^{5}\left\vert z\right\vert ^{5}}. \label{c01}%
\end{align}
On the other hand,
\begin{align}
\left\vert \int_{1<\left\vert z-y\right\vert <\frac{\left\vert z\right\vert
}{2}}\frac{\rho\left(  \left\vert z\right\vert \right)  -\rho\left(
\left\vert y\right\vert \right)  }{\left\vert z-y\right\vert ^{5}%
}dy\right\vert  &  \leq\frac{C\left\Vert \rho\right\Vert _{\beta
_{0}+2,1;\varepsilon}}{\left(  1+\varepsilon\left\vert z\right\vert \right)
^{\beta_{0}+2}}\int_{1<\left\vert z-y\right\vert <\frac{\left\vert
z\right\vert }{2}}\frac{dy}{\left\vert z-y\right\vert ^{5}}\label{c2}\\
&  \leq\frac{C\left\Vert \rho\right\Vert _{\beta_{0}+2,1;\varepsilon}}{\left(
1+\varepsilon\left\vert z\right\vert \right)  ^{\beta_{0}+2}}.\nonumber
\end{align}
Furthermore, using the fact that $\rho\in C^{1,\alpha},$ we get
\begin{equation}
\left\vert \text{P.V.}\int_{0<\left\vert z-y\right\vert <1}\frac{\rho\left(
\left\vert z\right\vert \right)  -\rho\left(  \left\vert y\right\vert \right)
}{\left\vert z-y\right\vert ^{5}}dy\right\vert \leq C\left\Vert \rho
\right\Vert _{C^{1,\alpha}\left(  B_{1}\left(  z\right)  \right)  }.
\label{c3}%
\end{equation}
Inequalities $\left(  \ref{c01}\right)  ,\left(  \ref{c2}\right)  ,\left(
\ref{c3}\right)  $ give us the required weighted $C^{0}$ estimate of $\left(
\left\vert \xi\right\vert \hat{\rho}\left(  \xi\right)  \right)  ^{\vee
}\left(  z\right)  .$ Similarly, one can also get corresponding estimate for
the Holder norm. Hence the desired estimate $\left(  \ref{equ2}\right)  $ follows.

To find the solution $f$ for the second equation in $\left(  \ref{fg1}\right)
,$ we first consider the equation
\begin{equation}
\left(  f^{\prime\prime}+\frac{3}{l}f^{\prime}+\left\vert A\right\vert
^{2}\right)  ^{\symbol{94}}=\frac{2\hat{f}\left(  \xi\right)  +\hat{\gamma
}_{-1}\left(  \xi\right)  +\hat{\gamma}_{1}\left(  \xi\right)  }{2p_{2,\xi
}^{\prime}\left(  1\right)  }. \label{fg3}%
\end{equation}
This can be written as
\begin{equation}
\hat{f}\left(  \xi\right)  =-\frac{\left(  \left\vert A\right\vert
^{2}\right)  ^{\symbol{94}}}{\left\vert \xi\right\vert ^{2}-\frac{1}{p_{2,\xi
}^{\prime}\left(  1\right)  }}+\frac{\hat{\gamma}_{-1}\left(  \xi\right)
+\hat{\gamma}_{1}\left(  \xi\right)  }{2\left(  \left\vert \xi\right\vert
^{2}p_{2,\xi}^{\prime}\left(  1\right)  -1\right)  }. \label{fg2}%
\end{equation}
We may take inverse Fourier transform on both sides of $\left(  \ref{fg2}%
\right)  .$ Let
\[
K_{1}=\left(  \frac{1}{\left\vert \xi\right\vert ^{2}-\frac{1}{p_{2,\xi
}^{\prime}\left(  1\right)  }}\right)  ^{\vee},\text{ }K_{2}=\left(  \frac
{1}{\left\vert \xi\right\vert ^{2}p_{2,\xi}^{\prime}\left(  1\right)
-1}\right)  ^{\vee}.
\]
In view of the explicit formula of $p_{2,\xi}^{\prime}\left(  1\right)  ,$ we
know $\left\vert \xi\right\vert ^{2}-\frac{1}{p_{2,\xi}^{\prime}\left(
1\right)  }$ and $\left\vert \xi\right\vert ^{2}p_{2,\xi}^{\prime}\left(
1\right)  -1$ are positive and real analytic. This implies that $K_{1}$ and
$K_{2}$ decay fast enough at infinity. On the other hand,
\[
\frac{1}{\left\vert \xi\right\vert ^{2}-\frac{1}{p_{2,\xi}^{\prime}\left(
1\right)  }}\sim\frac{1}{\left\vert \xi\right\vert ^{2}},\text{ }\frac
{1}{\left\vert \xi\right\vert ^{2}p_{2,\xi}^{\prime}\left(  1\right)  -1}%
\sim\frac{1}{\left\vert \xi\right\vert },\text{ as }\left\vert \xi\right\vert
\rightarrow+\infty.
\]
Observe that the inverse Fourier transform of $\left\vert \xi\right\vert
^{-1}$ is $c_{1}\left\vert x\right\vert ^{-3}($see \cite{G}$).$ It follows
that $K_{2}$ has a singularity of the order $O\left(  \left\vert x\right\vert
^{-3}\right)  $ near origin. The estimate $\left(  \ref{equ1}\right)  $ for
solution $f$ of $\left(  \ref{fg2}\right)  $ then follows from routine
calculation in potential theory. Since by Lemma \ref{Lap}, $\Delta_{\Gamma
_{0}}f$ is a small perturbation of $f^{\prime\prime}+\frac{3}{l}f^{\prime},$
then we can use a perturbation argument to show the same estimate for solution
$f$ of the second equation in $\left(  \ref{fg1}\right)  .$ This finishes the proof.
\end{proof}

With the model linear problem understood, we proceed to solve the nonlinear
problem. Let $\phi_{0}$ be the solution of the problem
\[
\left\{
\begin{array}
[c]{l}%
\partial_{t}^{2}\phi_{0}+\partial_{l}^{2}\phi_{0}+\frac{3}{l}\partial_{l}%
\phi_{0}=t\left\vert A\right\vert ^{2},\text{ in }\Omega_{h},\\
\phi_{0}=0\text{ on }\partial\Omega_{h}.
\end{array}
\right.
\]

\begin{lemma}
\label{decay}Suppose $\left\Vert f-f_{0}\right\Vert _{\beta_{0},2;\varepsilon
}\leq C\varepsilon^{2},\left\Vert \bar{g}\right\Vert _{\beta_{0}%
,2;1,\symbol{94}}\leq C\varepsilon$, and $\left\Vert \phi-\phi_{0}\right\Vert
_{\beta_{0},2;\ast}\leq C\varepsilon^{2}.$ There holds
\begin{align*}
\left\Vert E_{3,1}-E_{3,-1}\right\Vert _{\beta_{0}+2,1;\varepsilon}  &  \leq
C\varepsilon^{3},\\
\left\Vert E_{3,1}+E_{3,-1}\right\Vert _{\beta_{0},1;\varepsilon}  &  \leq
C\varepsilon^{3}.
\end{align*}

\end{lemma}

\begin{proof}
Recall that
\[
E_{3,i}=-\frac{1}{2}\left(  1+\mathfrak{g}^{1,1}h_{i}^{\prime2}\right)
\left(  \partial_{t}\phi\right)  ^{2}+\frac{\mathfrak{g}^{1,1}h_{i}^{\prime}%
}{1+f}\partial_{t}\phi+\frac{1}{2}f^{2}-\frac{1}{2}\mathfrak{g}^{1,1}\left(
g^{\prime}+tf^{^{\prime}}\right)  ^{2}.\text{ }%
\]
Using the boundedness of $\mathfrak{g}^{1,1},$ taking into account of the fact
that%
\[
\left\Vert g^{\prime}\right\Vert _{3,1;\varepsilon}\leq C\varepsilon
^{2},\left\Vert f\right\Vert _{2,2;\varepsilon}\leq C\varepsilon
^{2},\left\Vert \partial_{t}\phi\left(  \pm1,l\right)  -\partial_{t}\phi
_{0}\left(  \pm1,l\right)  \right\Vert _{\beta_{0},2;\varepsilon}\leq
C\varepsilon^{2},
\]
we find that
\[
\left\Vert \mathfrak{g}^{1,1}h_{i}^{\prime2}\left(  \partial_{t}\phi\right)
^{2}\right\Vert _{\beta_{0}+2,1;\varepsilon}+\left\Vert \mathfrak{g}%
^{1,1}\left(  g^{\prime}\right)  ^{2}\right\Vert _{\beta_{0}+2,1;\varepsilon
}+\left\Vert \mathfrak{g}^{1,1}g^{\prime}f\right\Vert _{\beta_{0}%
+2,1;\varepsilon}\leq C\varepsilon^{3}.
\]
Now we subtract $E_{3,1}$ with $E_{3,-1},$ the term $f^{2}$ will be cancelled.
Additionally, using the asymptotic expansion of $\mathfrak{g}^{1,1},$ we know
\[
\left\Vert \left(  t^{2}\mathfrak{g}^{1,1}f^{^{\prime}2}\right)
|_{t=-1}-\left(  t^{2}\mathfrak{g}^{1,1}f^{^{\prime}2}\right)  |_{t=1}%
\right\Vert _{\beta_{0}+2,1;\varepsilon}\leq C\varepsilon^{3}.
\]
Furthermore, observing that $\left\Vert f_{0}^{\prime}\right\Vert
_{3,1;\varepsilon}\leq C\varepsilon^{2},$ we get
\begin{align*}
&  \left\Vert \left(  \mathfrak{g}^{1,1}h_{-1}^{\prime}\partial_{t}%
\phi\right)  |_{t=-1}-\left(  \mathfrak{g}^{1,1}h_{1}^{\prime}\partial_{t}%
\phi\right)  |_{t=1}\right\Vert _{\beta_{0}+2,1;\varepsilon}\\
&  \leq C\varepsilon^{3}+C\left\Vert \left(  f_{0}^{\prime}\partial_{t}%
\phi_{0}\right)  |_{t=-1}+\left(  f_{0}^{\prime}\partial_{t}\phi_{0}\right)
|_{t=1}\right\Vert _{\beta_{0}+2,1;\varepsilon}\\
&  \leq C\varepsilon^{3}.
\end{align*}
Hence we get
\begin{align*}
\left\Vert E_{3,1}-E_{3,-1}\right\Vert _{\beta_{0}+2,1;\varepsilon}  &  \leq
C\varepsilon^{3}+\frac{1}{2}\left\Vert \partial_{t}\phi_{0}\left(
-1,l\right)  ^{2}-\partial_{t}\phi_{0}\left(  1,l\right)  ^{2}\right\Vert
_{\beta_{0}+2,1;\varepsilon}\\
&  \leq C\varepsilon^{3}.
\end{align*}
The proof of $\left\Vert E_{3,1}+E_{3,-1}\right\Vert _{\beta_{0}%
,1;\varepsilon}\leq C\varepsilon^{3}$ is similar.
\end{proof}

To proceed, let us consider the nonlinear problem
\begin{equation}
\left\{
\begin{array}
[c]{l}%
\Delta\phi=J_{\Gamma_{0}}g+\left(  \Delta_{\Gamma_{0}}f+\left\vert
A\right\vert ^{2}\right)  t-E_{1}+E_{2}+\Delta_{\Gamma_{0}}w-\Delta
_{\Gamma_{s}}w,\text{ in }\Omega_{h},\\
\phi=0\text{ on }\partial\Omega_{h}.
\end{array}
\right.  \label{nonlinear1}%
\end{equation}
Let us introduce the notation%
\begin{equation}
P\left(  f,\bar{g},\phi\right)  :=-E_{1}+E_{2}+\Delta_{\Gamma_{0}}%
w-\Delta_{\Gamma_{s}}w+\partial_{t}^{2}\phi+\partial_{l}^{2}\phi+\frac{3}%
{l}\partial_{l}\phi-\Delta\phi. \label{P}%
\end{equation}
We will investigate the Lipschitz dependence of $P$ on $f$ and $\bar{g}.$

\begin{lemma}
\label{P1}For $f_{i}\in\mathcal{B}_{2,2;\varepsilon},\bar{g}_{i}%
\in\mathcal{\bar{B}}_{\beta_{0},2;1},$ with $\left\Vert f_{i}-f_{0}\right\Vert
_{\beta_{0},2;\varepsilon}\leq C\varepsilon^{2},\left\Vert \bar{g}%
_{i}\right\Vert _{\beta_{0},2;1,\symbol{94}}\leq C\varepsilon,$ $i=1,2,$ we
have%
\[
\left\Vert P\left(  f_{1},\bar{g}_{1},\phi\right)  -P\left(  f_{2},\bar{g}%
_{2},\phi\right)  \right\Vert _{\beta_{0}+2,0;\varepsilon}=O\left(
\varepsilon^{2}\right)  \left\Vert f_{1}-f_{2}\right\Vert _{\beta
_{0},2;\varepsilon}+O\left(  \varepsilon^{3}\right)  \left\Vert \bar{g}%
_{1}-\bar{g}_{2}\right\Vert _{\beta_{0},2;1,\symbol{94}}.
\]

\end{lemma}

\begin{proof}
Let us consider the terms in $\left(  \ref{P}\right)  $. Recall that
\[
E_{1}\left(  f,\bar{g}\right)  =-tf\Delta_{\Gamma_{0}}f+\Delta_{\Gamma_{0}%
}\left(  fg\right)  -g\Delta_{\Gamma_{0}}f+\Delta_{\Gamma_{0}}\left[  \left(
s-g\right)  \frac{f^{2}}{1+f}\right]  .
\]
We compute directly that%
\begin{equation}
f_{1}\Delta_{\Gamma_{0}}f_{1}-f_{2}\Delta_{\Gamma_{0}}f_{2}=f_{1}%
\Delta_{\Gamma_{0}}\left(  f_{1}-f_{2}\right)  +\Delta_{\Gamma_{0}}%
f_{2}\left(  f_{1}-f_{2}\right)  . \label{e1}%
\end{equation}
Next, since
\[
\Delta_{\Gamma_{0}}\left(  fg\right)  -g\Delta_{\Gamma_{0}}f=2f^{\prime
}g^{\prime}+f\Delta_{\Gamma_{0}}g,
\]
we have%
\begin{align}
&  \left[  \Delta_{\Gamma_{0}}\left(  f_{1}g_{1}\right)  -g_{1}\Delta
_{\Gamma_{0}}f_{1}\right]  -\left[  \Delta_{\Gamma_{0}}\left(  f_{2}%
g_{2}\right)  -g_{2}\Delta_{\Gamma_{0}}f_{2}\right] \nonumber\\
&  =2\left(  f_{1}^{\prime}-f_{2}^{\prime}\right)  g_{1}^{\prime}%
+2f_{2}^{\prime}\left(  g_{1}^{\prime}-g_{2}^{\prime}\right) \nonumber\\
&  +\Delta_{\Gamma_{0}}g_{1}\left(  f_{1}-f_{2}\right)  +f_{2}\Delta
_{\Gamma_{0}}\left(  g_{1}-g_{2}\right)  . \label{e2}%
\end{align}
Now combining $\left(  \ref{e1}\right)  ,\left(  \ref{e2}\right)  $ and
performing a similar computation for the term $\Delta_{\Gamma_{0}}\left[
\left(  s-g\right)  \frac{f^{2}}{1+f}\right]  ,$ we obtain
\[
\left\Vert E_{1}\left(  f_{1},\bar{g}_{1}\right)  -E_{1}\left(  f_{2},\bar
{g}_{2}\right)  \right\Vert _{\beta_{0}+2,0;\varepsilon}=O\left(
\varepsilon^{2}\right)  \left\Vert f_{1}-f_{2}\right\Vert _{\beta
_{0},2;\varepsilon}+O\left(  \varepsilon^{3}\right)  \left\Vert \bar{g}%
_{1}-\bar{g}_{2}\right\Vert _{\beta_{0},2;1}.
\]

For the term
\[
E_{2}\left(  f,g\right)  =\frac{1}{1+f}\sum\limits_{i=1}^{6}\frac{s^{2}%
k_{i}^{3}}{1-sk_{i}}-\frac{fg\left\vert A\right\vert ^{2}}{1+f},
\]
we have%
\[
E_{2}\left(  f_{1},g_{1}\right)  -E_{2}\left(  f_{2},g_{2}\right)
=-\left\vert A\right\vert ^{2}\left(  \frac{f_{1}g_{1}}{1+f_{1}}-\frac
{f_{2}g_{2}}{1+f_{2}}\right)  +\frac{f_{2}-f_{1}}{(1+f_{1})(1+f_{2})}%
\sum_{i=1}^{6}\frac{s^{2}k_{i}^{3}}{1-sk_{i}}.
\]
Since $\left\vert A\right\vert ^{2}=O\left(  \frac{\varepsilon^{2}}{\left(
1+\varepsilon l\right)  ^{2}}\right)  ,$ we obtain
\[
\left\Vert E_{2}\left(  f_{1},\bar{g}_{1}\right)  -E_{2}\left(  f_{2},\bar
{g}_{2}\right)  \right\Vert _{\beta_{0}+2,0;\varepsilon}=O\left(
\varepsilon^{2}\right)  \left\Vert f_{1}-f_{2}\right\Vert _{\beta
_{0},2;\varepsilon}+O\left(  \varepsilon^{3}\right)  \left\Vert \bar{g}%
_{1}-\bar{g}_{2}\right\Vert _{\beta_{0},2;1}.
\]

It remains to analyze the term $\Delta_{\Gamma_{0}}w-\Delta_{\Gamma_{s}}w.$ To
handle it, we simply note that by Lemma $\ref{Lap}$ the following expansion
holds:
\begin{align*}
\Delta_{\Gamma_{0}}w-\Delta_{\Gamma_{s}}w  &  =O\left(  \frac{\varepsilon^{2}%
}{\left(  1+\varepsilon l\right)  ^{2}}\right)  \partial_{l}w+O\left(
\frac{\varepsilon}{1+\varepsilon l}\right)  \partial_{l}^{2}w\\
&  =O\left(  \frac{\varepsilon^{2}}{\left(  1+\varepsilon l\right)  ^{2}%
}\right)  \left(  \frac{-g^{\prime}\left(  1+f\right)  -\left(  s-g\right)
f^{\prime}}{\left(  1+f\right)  ^{2}}\right) \\
&  +O\left(  \frac{\varepsilon}{1+\varepsilon l}\right)  \left(
\frac{-g^{\prime}\left(  1+f\right)  -\left(  s-g\right)  f^{\prime}}{\left(
1+f\right)  ^{2}}\right)  ^{\prime},
\end{align*}
which yields the desired estimate:
\begin{align*}
&  \left\Vert \left(  \Delta_{\Gamma_{0}}w-\Delta_{\Gamma_{s}}w\right)
|_{\left(  f_{1},g_{1}\right)  }-\left(  \Delta_{\Gamma_{0}}w-\Delta
_{\Gamma_{s}}w\right)  |_{\left(  f_{2},g_{2}\right)  }\right\Vert _{\beta
_{0}+2,0;\varepsilon}\\
&  =O\left(  \varepsilon^{2}\right)  \left\Vert f_{1}-f_{2}\right\Vert
_{\beta_{0},2;\varepsilon}+O\left(  \varepsilon^{3}\right)  \left\Vert \bar
{g}_{1}-\bar{g}_{2}\right\Vert _{\beta_{0},2;1}.
\end{align*}
The proof is thus completed.
\end{proof}

\begin{lemma}
\label{f}Given $f,\bar{g},$ with $\left\Vert f-f_{0}\right\Vert _{\beta
_{0},2;\varepsilon}\leq C\varepsilon^{2},\left\Vert \bar{g}\right\Vert
_{\beta_{0},2;1,\symbol{94}}\leq C\varepsilon$, problem $\left(
\ref{nonlinear1}\right)  $ has a unique solution $\phi$ with
\[
\left\Vert \phi-\phi_{0}\right\Vert _{\beta_{0}+1,2;\ast}\leq C\varepsilon
^{2}.
\]
If we write this solution as $\Phi\left(  f,\bar{g}\right)  ,$ then
\[
\left\Vert \Phi\left(  f_{1},\bar{g}_{1}\right)  -\Phi\left(  f_{2},\bar
{g}_{2}\right)  \right\Vert _{\beta_{0}+1,2;\ast}\leq C\left\Vert f_{1}%
-f_{2}\right\Vert _{\beta_{0},2;\varepsilon}+C\varepsilon^{2}\left\Vert
\bar{g}_{1}-\bar{g}_{2}\right\Vert _{\beta_{0},2;1,\symbol{94}}.
\]

\end{lemma}

\begin{proof}
We may recast $\left(  \ref{nonlinear1}\right)  $ as
\[
\phi=L_{1}\left[  J_{\Gamma_{0}}g+\left(  \Delta_{\Gamma_{0}}f+\left\vert
A\right\vert ^{2}\right)  t\right]  +L_{1}\left[  P\left(  f,\bar{g}%
,\phi\right)  \right]  ,
\]
where $\phi=\phi_{0}+\phi^{\ast},\phi^{\ast}\in\mathcal{B}_{\beta_{0}%
+1,2;\ast}.$ In other words,
\[
\phi^{\ast}=\bar{L}_{1}\left(  f,\bar{g},\phi^{\ast}\right)  :=L_{1}\left[
J_{\Gamma_{0}}g+\left(  \Delta_{\Gamma_{0}}f+\left\vert A\right\vert
^{2}\right)  t\right]  +L_{1}\left[  P\left(  f,\bar{g},\phi_{0}+\phi^{\ast
}\right)  \right]  -\phi_{0},
\]
We regard it as a fixed point problem of $\phi^{\ast}$ for the map $\bar
{L}_{1}.$ Observe that although $\phi_{0}$ only belongs to $\mathcal{B}%
_{2,2;\ast},$ the function $P\left(  f,\bar{g},\phi_{0}+\phi^{\ast}\right)  $
actually lies in $\mathcal{B}_{\beta_{0}+1,0;\ast}.$ Now we show $\bar{L}_{1}$
is a contraction map. Indeed, by Lemma \ref{Lap},
\begin{align*}
\Delta\phi &  =\partial_{s}^{2}\phi+\Delta_{\Gamma_{s}}\phi-H_{\Gamma_{s}%
}\partial_{s}\phi\\
&  =\frac{1}{\left(  1+f\right)  ^{2}}\partial_{t}^{2}\phi+\Delta_{\Gamma_{0}%
}\phi+O\left(  \frac{\varepsilon}{\left(  1+\varepsilon l\right)  ^{2}%
}\right)  \partial_{l}\phi\\
&  +O\left(  \frac{\varepsilon}{1+\varepsilon l}\right)  \partial_{l}^{2}%
\phi+O\left(  \sum k_{i}^{2}\right)  \partial_{t}\phi.
\end{align*}
Using this expansion, we can verify that
\[
\left\Vert \bar{L}_{1}\left(  f,\bar{g},\phi_{1}^{\ast}\right)  -\bar{L}%
_{1}\left(  f,\bar{g},\phi_{2}^{\ast}\right)  \right\Vert _{\beta_{0}%
+1,2;\ast}\leq C\varepsilon\left\Vert \phi_{1}^{\ast}-\phi_{2}^{\ast
}\right\Vert _{\beta_{0}+1,2;\ast}.
\]
This implies that $\bar{L}_{1}$ is a contraction mapping provided that
$\varepsilon$ is small enough. It follows that $\left(  \ref{nonlinear1}%
\right)  $ has a solution.

To see the Lipschitz dependence of $\Phi$ on $f,\bar{g},$ we subtract the
equations satisfied by $\Phi\left(  f_{1},\bar{g}_{1}\right)  $ and
$\Phi\left(  f_{2},\bar{g}_{2}\right)  $. Then one can use the explicit
expression for $E_{1},E_{2}$ to get the desired estimate.
\end{proof}

If we write $\Phi\left(  f,\bar{g}\right)  =\phi_{1}+L_{1}\left(  P\left(
f,\bar{g},\Phi\left(  f,\bar{g}\right)  \right)  \right)  ,$ then our original
nonlinear problem will be transformed into
\begin{equation}
\left\{
\begin{array}
[c]{l}%
\partial_{t}^{2}\phi_{1}+\partial_{l}^{2}\phi_{1}+\frac{3}{l}\partial_{l}%
\phi_{1}=J_{\Gamma_{0}}g+\left(  \Delta_{\Gamma_{0}}f+\left\vert A\right\vert
^{2}\right)  t+P(f,\bar{g},\phi_{1}),\text{ in }\Omega_{h},\\
\phi_{1}=0\text{ and }\partial_{t}\phi_{1}-f=E_{3,i}-\partial_{t}\left[
L_{1}\left(  P\left(  f,\bar{g},\Phi\left(  f,\bar{g}\right)  \right)
\right)  \right]  ,\text{on }\Gamma_{i+h_{i}}.
\end{array}
\right.  \label{Fi1}%
\end{equation}
With all these preparations, we are now ready to prove Theorem \ref{simons}.

\begin{proof}
[Proof of Theorem \ref{simons} ]Let us set $f=f_{0}+\widetilde{f}.$ Using
Proposition \ref{L2}, we find that to solve $\left(  \ref{Fi1}\right)  ,$ it
suffices to get a solution for the following fixed point problem for $\left(
\widetilde{f},g\right)  $ :
\[
\left(  \widetilde{f},\bar{g}\right)  =\bar{L}_{2}\left(  \widetilde{f}%
,\bar{g}\right)  :=L_{2}\left(  \Upsilon_{-1},\Upsilon_{1}\right)  -\left(
f_{0},0\right),
\]
where
\[
\Upsilon_{i}=E_{3,i}-\partial_{t}\left[  L_{1}\left(  P\left(  f,\bar{g}%
,\Phi\left(  f,\bar{g}\right)  \right)  \right)  \right]  |_{t=i},i=\pm1.
\]

Let us define the space
\[
\mathcal{B}:\mathcal{=}\left\{  \left(  \widetilde{f},\bar{g}\right)
|,\left(  \widetilde{f},\bar{g}\right)  \in\mathcal{B}_{\beta_{0}%
,2;\varepsilon}\times\mathcal{B}_{\beta_{0},2;1,\symbol{94}}\right\}  ,
\]
equipped with the norm
\[
\left\Vert \left(  \widetilde{f},\bar{g}\right)  \right\Vert :=\varepsilon
\left\Vert \widetilde{f}\right\Vert _{\beta_{0},2;\varepsilon}+\varepsilon
^{2}\left\Vert \bar{g}\right\Vert _{\beta_{0},2;1,\symbol{94}}.
\]
We claim that $\bar{L}_{2}$ is a contraction mapping in the set
\[
B_{1}:=\left\{  \left(  \widetilde{f},\bar{g}\right)  \in\mathcal{B}%
:\left\Vert \left(  \widetilde{f},\bar{g}\right)  \right\Vert \leq
C_{0}\varepsilon^{3}\right\}  ,
\]
where $C_{0}$ is a fixed large constant. Indeed, let
\[
\eta_{\pm}\left(  f,\bar{g}\right)  :=\partial_{t}\left[  L_{1}\left(
P\left(  f,\bar{g},\Phi\left(  f,\bar{g}\right)  \right)  \right)  \right]
|_{t=-1}\pm\partial_{t}\left[  L_{1}\left(  P\left(  f,\bar{g},\Phi\left(
f,\bar{g}\right)  \right)  \right)  \right]  |_{t=1},
\]
and
\[
f_{1}=f_{0}+\widetilde{f}_{1},f_{2}=f_{0}+\widetilde{f}_{2}.
\]
Using Proposition \ref{L2}, we can show%
\begin{align*}
&  \left\Vert \eta_{+}\left(  f_{1},\bar{g}_{1}\right)  -\eta_{+}\left(
f_{2},\bar{g}_{2}\right)  \right\Vert _{\beta_{0},2;\varepsilon}+\left\Vert
\eta_{-}\left(  f_{1},\bar{g}_{1}\right)  -\eta_{-}\left(  f_{2},\bar{g}%
_{2}\right)  \right\Vert _{\beta_{0},2;\varepsilon}\\
&  =O\left(  \varepsilon^{2}\right)  \left\Vert f_{1}-f_{2}\right\Vert
_{\beta_{0},2;\varepsilon}+O\left(  \varepsilon^{3}\right)  \left\Vert \bar
{g}_{1}-\bar{g}_{2}\right\Vert _{\beta_{0},2;1,\symbol{94}}.
\end{align*}
It then follows from Proposition \ref{L2}, Lemma \ref{P1} and Lemma \ref{f}
that
\[
\left\Vert \bar{L}_{2}\left(  \widetilde{f}_{1},\bar{g}_{1}\right)  -\bar
{L}_{2}\left(  \widetilde{f}_{2},\bar{g}_{2}\right)  \right\Vert \leq
C\varepsilon\left\Vert \left(  \widetilde{f}_{1},\bar{g}_{1}\right)  -\left(
\widetilde{f}_{2},\bar{g}_{2}\right)  \right\Vert .
\]
This proves the claim.

To prove the existence of a fixed point for $\bar{L}_{2},$ it remains to show
that $\bar{L}_{2}\left(  B_{1}\right)  \subset B_{1}.$ Since $\left(
\widetilde{f},\bar{g}\right)  \in B_{1},$ we have $\left\Vert \widetilde{f}%
\right\Vert _{\beta_{0},2;\varepsilon}\leq C_{0}\varepsilon^{2},\left\Vert
\bar{g}\right\Vert _{\beta_{0},2;1,\symbol{94}}\leq C_{0}\varepsilon.$ Observe
that due to the presence of the term $\left\vert A\right\vert ^{2}t$ and
$t^{3}\sum k_{i}^{3},$ the function $L_{1}\left(  P\left(  f,\bar{g}%
,\Phi\left(  f,\bar{g}\right)  \right)  \right)  |_{\pm1}$ does not have
enough decay and only belongs to $\mathcal{B}_{2,2;\varepsilon,\ast}.$
However, since these two terms are odd, their contribution to the boundary
derivative at $t=\pm1$ cancel and therefore
\[
\left\Vert \eta_{+}\right\Vert _{\beta_{0},2;\varepsilon}\leq C\varepsilon
^{2},\left\Vert \eta_{-}\right\Vert _{\beta_{0}+2,2;\varepsilon}\leq
C\varepsilon^{3}.
\]
Hence by Propositon \ref{L2},
\[
\bar{L}_{2}\left(  \widetilde{f},\bar{g}\right)  \leq C\varepsilon^{3},
\]
which implies that $\bar{L}_{2}\left(  B_{1}\right)  \subset B_{1}$, provided
that $C_{0}$ is chosen large enough.

The solution $w_{h}+\phi$ depends smoothly on $\varepsilon.$ Let us take the
derivatives of $w_{h}+\phi$ with respect to $\varepsilon.$ Note that the main
order of $w_{h}+\phi$ is $\frac{s-g}{1+f},$ where $s$ is the Fermi coordinate
around the minimal hypersurface $S_{\varepsilon}.$ Using the fact that
$S_{\varepsilon}$ is a minimal foliation associated to the Simons' cone, we
find that $\frac{d\left(  w_{h}+\phi\right)  }{d\varepsilon}$ is positive and
satisfy the system $\left(  \ref{stable}\right)  $ (see \cite{Jerison1}). This
proves that our solution of the free boundary problem is stable. This finishes
the proof of Theorem \ref{simons}.
\end{proof}

\section{\bigskip Existence of an energy minimizer in $\mathbb{R}^{8}$---Proof
of Theorem \ref{mini}}

In the previous section, we have shown that if $\varepsilon_{0}>0$ is small
enough, then for each $\varepsilon<\varepsilon_{0},$ we have a solution for
the free boundary problem whose nodal set is asymptotic to $S_{\varepsilon
}^{+}.$ By symmetry, one also has solutions whose nodal sets are asymptotic to
$S_{\varepsilon}^{-}.$ We denote these two continuous families of solutions by
$u_{\varepsilon}^{+}$ and $u_{\varepsilon}^{-},$ with $u_{\varepsilon}%
^{-}<u_{\varepsilon}^{+}.$ In this section, we will use variational arguments
to show the existence of an energy minimizer $U$ in $\mathbb{R}^{8},$ lying
between $u_{\varepsilon_{0}}^{+}$ and $u_{\varepsilon_{0}}^{-}.$ The arguments
in this section are very similar to that of \cite{Liu}, where the global
minimizers of the Allen-Cahn equation in dimension $n\geq8$ are constructed.

We use $B_{a}$ to denote the open ball of radius $a$ in $\mathbb{R}^{8}.$
Choose a Lipschitz function $b_{a}$ which is invariant under the natural
$O\left(  4\right)  \times O\left(  4\right)  $ action on $\mathbb{R}^{8}$ and%
\[
u_{\varepsilon_{0}}^{-}<b_{a}<u_{\varepsilon_{0}}^{+}\text{ on }\partial
B_{a}.
\]
Let us consider the minimizing problem
\begin{equation}
\min_{\eta-b_{a}\in H_{0}^{1}\left(  B_{a}\right)  }J\left(  \eta\right)  .
\label{min}%
\end{equation}

\begin{lemma}
\label{minimizer}The minimizing problem $\left(  \ref{min}\right)  $ has a
solution $u_{a}$ which is invariant under $O\left(  4\right)  \times O\left(
4\right)  .$
\end{lemma}

\begin{proof}
The existence of a minimizer $u$ for $\left(  \ref{min}\right)  $ follows from
standard arguments. The point is that we need to prove the existence of a
minimizer which is additionally invariant under $O\left(  4\right)  \times
O\left(  4\right)  .$

Since $u$ solves the free boundary problem, it is continuous. We define%
\begin{align*}
w_{1}\left(  x\right)   &  =\min\left\{  u\left(  gx\right)  :g\in O\left(
4\right)  \times O\left(  4\right)  \right\}  ,\\
w_{2}\left(  x\right)   &  =\max\left\{  u\left(  gx\right)  :g\in O\left(
4\right)  \times O\left(  4\right)  \right\}  .
\end{align*}
Then $w_{1}$ and $w_{2}$ are invariant under $O\left(  4\right)  \times
O\left(  4\right)  .$ We claim that $w_{1}$ and $w_{2}$ are also minimizers.
Indeed, for each $k\in\mathbb{N}$ and a finite set $\{g_{1},\cdots,g_{k}\}\in
O\left(  4\right)  \times O\left(  4\right)  $, let
\[
\bar{w}_{k}=\min\left\{  u\left(  g_{i}x\right)  :g_{i}\in O\left(  4\right)
\times O\left(  4\right)  ,i=1,...,k\right\}.
\]
Then $\bar{w}_{k}$ is a minimizer. We cover $O\left(  4\right)  \times O\left(
4\right)  $ by finitely many balls with radius $\varepsilon.$ Denote by
$n_{\varepsilon}$ the number of balls. In each ball, let us choose a $g_{i}\in
O\left(  4\right)  \times O\left(  4\right)  $. We will define
\[
q_{\varepsilon}\left(  x\right)  :=\min\left\{  u\left(  g_{i}x\right)
:i=1,...,n_{\varepsilon}\right\}  .
\]
Then $q_{\varepsilon}$ is also a minimizer. We observe that by the continuity
of a minimizer,
\[
w_{1}\left(  x\right)  =\lim_{\varepsilon\rightarrow0}q_{\varepsilon}\left(
x\right)  .
\]
On the other hand, let $\left\{  \varepsilon_{k}\right\}  $ be a sequence
converge to $0.$ Then standard arguments yield that $q_{\varepsilon_{k}%
}\left(  x\right)  $ converges a.e. to minimizer $q.$ This $q$ must be
$w_{1}.$ This proves that $w_{1}$ is also a minimizer. Similarly, $w_{2}$ is
also a minimizer.
\end{proof}

\subsection{Regularity of the free boundary}

We would like to analyze the regularity property of the free boundary of the
solution $u_{a}.$

\begin{lemma}
The free boundary of $u_{a}$ is smooth in $B_{a}\backslash\left\{  0\right\}
.$
\end{lemma}

\begin{proof}
We shall use the standard arguments in the regularity theory: Blow up analysis
around a free boundary point, cf. \cite{Weiss2, Weiss3}. Let $x_{0}\in B_{a}$
be a point on the free boundary of $u$. Suppose $x_{0}\neq0$ and $u_{a}\left(
x_{0}\right)  =1.$ We distinguish three cases.

Case 1. $x_{0}$ is not on the $x$ axis and not on $y$ axis.

In this case, standard arguments, based on Weiss monotonicity formula
(\cite{Weiss2, Weiss3}), tell us that the sequence $w_{k}:=\frac{u_{a}\left(
x_{0}+\rho_{k}\cdot\right)  -1}{\rho_{k}},$ with $\rho_{k}\rightarrow0,$ has a
subsequence converges in suitable sense to a minimizing cone $\mathfrak{C}$ in
$\mathbb{R}^{8}.$ We observe that $u_{a}$ is invariant under $O\left(
4\right)  \times O\left(  4\right)  .$ Hence $\mathfrak{C}$ reduces to a
minimizing cone in $\mathbb{R}^{2}.$ Therefore it must be a trivial cone. This
implies that around $x_{0}$, the free boundary is flat and the regularity
theory implies that actually it is smooth (analytic).

Case 2. $x_{0}$ is on the $x$ or $y$ axis.

In this case, the cone $\mathfrak{C}$ reduces to a minimizing cone in
$\mathbb{R}^{5}$ which is invariant under the $O\left(  4\right)  $ action of
the last four coordinates. If this cone were not trivial, it would be
unstable, due to the classification of stable cones by Jerison and Savin in
the axial symmetric case (see \cite{Jerison1}). This contradicts with the fact
that $u_{a}$ is a minimizer.
\end{proof}

With this regularity at hand, we now want to prove that these minimizers are
bounded by $u_{\varepsilon_{0}}^{+}$ and $u_{\varepsilon_{0}}^{-},$ by
sweeping the family of ordered solutions $u_{\varepsilon}^{+}$ and
$u_{\varepsilon}^{-}$, similarly as in $\cite{Liu}.$ By our previous
construction, for $\varepsilon$ sufficiently small, we have
\begin{equation}
\left\{
\begin{array}
[c]{l}%
u_{a}\leq u_{\varepsilon}^{+},\text{ in }B_{a},\\
u_{a}<u_{\varepsilon}^{+},\text{ in }\Lambda:=\left\{  X:\left\vert
u_{a}\left(  X\right)  \right\vert <1\right\}  .
\end{array}
\right.  \label{order}%
\end{equation}
We show that actually $\left(  \ref{order}\right)  $ holds for all
$\varepsilon\leq\varepsilon_{0}.$ To see this, we continuously increase the
value of $\varepsilon$. Assume to the contrary that there existed a
$\delta<\varepsilon_{0},$ which were the first value where we have
\begin{equation}
u_{a}\leq u_{\delta}^{+}\text{ in }B_{a},\text{ and }u_{a}\left(  X\right)
=u_{\varepsilon}^{+}\left(  X\right)  \text{ for some }X\in\bar{\Lambda}.
\label{equal}%
\end{equation}
Maximum principle tells us that this $X$ must be on $\partial B_{a}.$ By the
results in \cite{KKS}, the free boundary approaches the fixed boundary
tangentially, this contradicts with the choice of $\delta,$ which is the
smallest value satisfying $\left(  \ref{equal}\right)  .$ This finishes the proof.

\begin{proof}
[Proof of Theorem 2]For each $a$ large, we have a solution $u_{a}$ with
$u_{\varepsilon_{0}}^{-}<u_{a}<u_{\varepsilon_{0}}^{+}.$ Sending $a$ to
infinity, we can find a subsequence of $u_{a}$ which converges to a nontrivial
solution $U$ of $\left(  \ref{FB}\right)  .$ This solution $U$ must be an
energy minimizer of $J,$ since each $u_{a}$ is minimizing.
\end{proof}

\section{From minimizers in $\mathbb{R}^{8}$ to monotone solutions in
$\mathbb{R}^{9}$--Proof of Theorem \ref{mono}}

We have obtained a minimizer of the energy functional in dimension 8. Now we
would like to construct monotone solutions in $\mathbb{R}^{9}$ from $U$,
following the arguments of Jerison-Monneau (\cite{JM}). We use $\left(
x^{\prime},x_{9}\right)  $ to denote the coordinate of a point in
$\mathbb{R}^{9},$ where $x^{\prime}\in$ $\mathbb{R}^{8}.$ We will still use
minimizing argument and work directly in the class of functions which is
invariant w.r.p.t $O\left(  4\right)  \times O\left(  4\right)  $ action on
the first eight variables.

We denote by $v_{1}$ the global minimizer in $\mathbb{R}^{8}$ we constructed
in the last section. We also consider the solution $v_{2}$ which in the
$\left(  x,y\right)  $ coordinate is given by
\[
v_{2}\left(  x,y\right)  =-v_{1}\left(  y,x\right)  .
\]
Since $v_{1}$ is constructed using minimizing argument, we can assume without
loss of generality that $v_{1}\leq v_{2}.$

\begin{proposition}
\label{monotone}Either there exists a nontrivial solution $u:\mathbb{R}%
^{9}\rightarrow\mathbb{R}$ monotone in the $x_{9}$ direction, or for each
$\delta\in\left[  v_{1}\left(  0\right)  ,v_{2}\left(  0\right)  \right]  ,$
there exists a nontrivial global minimizer $v$ in $\mathbb{R}^{8}$ with
$v\left(  0\right)  =\delta.$
\end{proposition}

\begin{proof}
Let $\rho$ be a smooth decreasing cutoff function which satisfies
\[
\rho\left(  s\right)  =\left\{
\begin{array}
[c]{c}%
1,s<1,\\
0,s>2.
\end{array}
\right.
\]
Define the function $w\left(  x^{\prime},x_{9}\right)  =\rho\left(
x_{9}\right)  v_{1}\left(  x^{\prime}\right)  +\left(  1-\rho\left(
x_{9}\right)  \right)  v_{2}\left(  x^{\prime}\right)  .$ For each cylinder
$C_{R^{\prime},l}=B_{R^{\prime}}\times\left[  -l,l\right]  ,$ consider the
minimization problem which equals $w$ on $\partial B_{R^{\prime}}\times\left[
-l,l\right]  $ and equals $v_{1}$ on $B_{R^{\prime}}\times\left\{  -l\right\}
,$ equals $v_{2}$ on $B_{R^{\prime}}\times\left\{  l\right\}  ,$ in the class
of functions which are invariant under $O\left(  4\right)  \times O\left(
4\right)  $ with respect to the first eight variables. We can find a minimizer
$u_{R^{\prime},l}$ that is monotone in the $x_{9}$ direction with this
boundary condition. By the gradient bound of De Silva-Jerison (\cite{Jerison3}%
), the free boundary is smooth in the interior of the cylinder.

Let $l\rightarrow+\infty,$ we get a solution $u_{R^{\prime}}$ on the whole
cylinder $B_{R^{\prime}}\times\mathbb{R},$ still monotone in $x_{9}$ and
invariant under $O\left(  4\right)  \times O\left(  4\right)  .$ We observe
that
\begin{equation}
\lim_{x_{9}\rightarrow+\infty}u_{R^{\prime}}=v_{2},\lim_{x_{9}\rightarrow
-\infty}u_{R^{\prime}}=v_{1}, \label{limit}%
\end{equation}
otherwise it will contradict with the fact that $v_{1}$ and $v_{2}$ are global
minimizer. Now fix an $a\in\left(  v_{1}\left(  0\right)  ,v_{2}\left(
0\right)  \right)  .$ By $\left(  \ref{limit}\right)  ,$ there exists
$h_{R^{\prime}}$ such that
\[
u_{R^{\prime}}\left(  x^{\prime},h_{R^{\prime}}\right)  =a.
\]
Let $\bar{u}_{R^{\prime}}\left(  x^{\prime},x_{9}\right)  =u_{R^{\prime}%
}\left(  x^{\prime},x_{9}-h_{R^{\prime}}\right)  .$ Then $\bar{u}_{R^{\prime}%
}\left(  x^{\prime},0\right)  =a.$ Let $R^{\prime}\rightarrow+\infty,$ we get
a solution $u$ monotone in $x_{9},$ invariant under $O\left(  4\right)  \times
O\left(  4\right)  ,$ and
\[
u\left(  0\right)  =a,v_{1}\leq u\leq v_{2}.
\]
If $u$ is independent on $x_{9},$ then $u$ is a global minimizer in
$\mathbb{R}^{8}.$\ This proves the proposition.
\end{proof}

\ \ Finally we are ready to prove Theorem \ref{mono}.

\begin{theorem}
There exists a solution $u$ to our free boundary problem such that $u$ is
invariant w.r.p.t $O\left(  4\right)  \times O\left(  4\right)  ,$ monotone in
$x_{9}$ and $u$ is not one dimensional.
\end{theorem}

\begin{proof}
Suppose the second possibility of Proposition \ref{monotone} occurs. Then we
can assume there is a global minimizer $v$ in $\mathbb{R}^{8},$ invariant
under $O\left(  4\right)  \times O\left(  4\right)  $ and $-1<v\left(
0\right)  <1.$

By $\Theta$ we shall denote the standard one dimensional solution to our free
boundary problem:
\[
\Theta\left(  x\right)  =\left\{
\begin{array}
[c]{c}%
x,x\in\left[  -1,1\right]  ,\\
1,x>1,\\
-1,x<-1.
\end{array}
\right.
\]
Note that $\Theta$ is monotone, but not strictly monotone. We would like to
pose suitable boundary condition on the cylinder $C_{R^{\prime},l}.$ For each
$t\in\left[  0,1\right]  ,$ let
\[
\Theta_{t}\left(  x^{\prime},x_{9}\right)  =\Theta\left(  tv\left(  x^{\prime
}\right)  +\left(  1-t\right)  x_{9}\right)  .
\]
Then $\Theta_{1}\left(  x^{\prime},x_{9}\right)  =\Theta\left(  v\left(
x^{\prime}\right)  \right)  =v\left(  x^{\prime}\right)  .$ $\Theta_{t}$ is a
connection between $\Theta$ and $v.$ Certainly, $\Theta_{t}\left(  x^{\prime
},x_{9}\right)  \in\left[  -1,1\right]  .$ We check that $\Theta_{t}$ is
continuous and monotone in the $x_{9}$ direction, since $\Theta$ itself is
monotone. Consider those points where
\begin{equation}
tv\left(  x^{\prime}\right)  +\left(  1-t\right)  x_{9}=1. \label{boundary}%
\end{equation}
For each fixed $x^{\prime},$ there is a unique point $x_{9}$ satisfying
$\left(  \ref{boundary}\right)  .$

Let $U_{t,R^{\prime},l}$ be the minimizer of $J$ in the symmetric (invariant
under $O\left(  4\right)  \times O\left(  4\right)  $ action) class of
functions defined on $C_{R^{\prime},l}$ with boundary condition
\[
U_{t}|_{\partial C_{R^{\prime},l}}=\Theta_{t}|_{\partial C_{R^{\prime},l}}.
\]
After a possible translation in the $x_{9}$ direction, we can assume that
\[
U_{t,R^{\prime},l}\left(  0\right)  =v\left(  0\right)  .
\]
For each $R^{\prime},$ letting $l\rightarrow+\infty,$ $U_{t,R^{\prime},l}$
converges pointwisely to a solution $U_{t,R^{\prime}}$, defined on the
infinite cylinder $C_{R^{\prime},+\infty}.$ $U_{t,R^{\prime}}$ is monotone in
$x_{9}$ on the boundary of $C_{R^{\prime},+\infty}.$ Then one can show that
$U_{t,R^{\prime}}$ is monotone in $x_{9}$ in $C_{R^{\prime},+\infty}$, with
\[
U_{t,R^{\prime}}\left(  0\right)  =v\left(  0\right)  .
\]

We claim that the map $t\rightarrow\partial_{x_{9}}U_{t,R^{\prime}}\left(
0\right)  $ is a continuous map. We first show that it is continuous at the
points where $t\neq1.$ In this case, let $t_{n}\rightarrow t.$ Then the
sequence $U_{t_{n},R^{\prime}}$ converges to a monotone solution $W.$ This $W$
must be equal to $U_{t,R^{\prime}}.$ Indeed, since $w$ and $U_{t,R^{\prime}}$
are equal to each other on the boundary of the cylinder and the boundary value
are monotone in the $x_{9}$ direction, we can infer that $W\geq U_{t,R^{\prime
}}$ and $W\leq U_{t,R^{\prime}}$ by the sliding method.

The continuity at $t=1$ also follows from similar arguments as that of
Jerison-Monneau \cite{JM}. The proof is thus completed.
\end{proof}

\section{Solutions from Catenoids}

In this section, we shall construct solutions of the free boundary problem
starting from another type of minimal surfaces---Catenoids. Since most of the
arguments are similar to the Simons' cone case, we will only sketch the proof
and point out the difference if necessary.

We remark that it is possible to do the construction for more general minimal
surfaces, but this is beyond the scope of this paper.

\subsection{The geometry of the catenoids}

To begin with, let us choose an \textquotedblleft arc-length\textquotedblright%
\ parametrization for the catenoid, this choice of coordinate will simplify
the computation. Let $\left(  x_{1},...,x_{n}\right)  $ be the coordinate in
$\mathbb{R}^{n}.$ Let $\left(  r,\theta\right)  $ be the polar coordinate in
$\mathbb{R}^{n-1},$ where $\theta$ is the coordinate on the unit sphere
$S^{n-2}$ in $\mathbb{R}^{n-1}.$ As we mentioned before, the generalized
catenoid $\mathcal{C}_{\varepsilon}$ in $\mathbb{R}^{n}$ can be described by
\[
x_{n}=\bar{\omega}_{\varepsilon}\left(  r\right)  ,r\in\lbrack r_{0}%
,+\infty).
\]
Introduce
\[
l=l\left(  r\right)  :=\int_{r_{0}}^{r}\sqrt{1+\bar{\omega}_{\varepsilon
}^{\prime}\left(  s\right)  ^{2}}ds.
\]
Then locally the catenoid can also be described by the coordinate $\left(
l,\theta\right)  .$ We would like to write the Laplacian-Beltrami operator
$\Delta_{\mathcal{C}_{\varepsilon}}$ on $\mathcal{C}_{\varepsilon}$ in this
coordinate$.$ In the $\left(  r,\theta\right)  $ variable, the metric tensor
on $\mathcal{C}$ is given by
\[
\left[  1+\bar{\omega}_{\varepsilon}^{\prime}\left(  r\right)  ^{2}\right]
dr^{2}+r^{2}d\theta^{2}.
\]
It follows that the metric $\mathtt{g}$ in the $\left(  l,\theta\right)  $
coordinate is $dl^{2}+r^{2}d\theta^{2}.$ Observe that $\det\mathtt{g}%
=r^{2\left(  n-2\right)  }$. For rotationally symmetric function
$\varphi=\varphi\left(  l\right)  ,$ the Laplacian-Beltrami operator is given
by
\begin{align}
\Delta_{\mathcal{C}_{\varepsilon}}\varphi &  =\frac{1}{\sqrt{\det\mathtt{g}}%
}\partial_{i}\left(  \sqrt{\det\mathtt{g}}\mathtt{g}^{ij}\partial_{j}%
\varphi\right) \nonumber\\
&  =\varphi^{\prime\prime}\left(  l\right)  +\frac{n-2}{r}\varphi^{\prime
}\left(  l\right) \nonumber\\
&  =\varphi^{\prime\prime}\left(  l\right)  +O\left(  \frac{\varepsilon
}{1+\varepsilon l}\right)  \varphi^{\prime}\left(  l\right)  . \label{lap2}%
\end{align}

Using $s$ to denote the signed distance of a point $P$ to $\mathcal{C}%
_{\varepsilon}.$ Then we can write
\[
P=X+s\nu\left(  X\right)  ,
\]
where $X=X\left(  l,\theta\right)  $ designates a point on the $\mathcal{C}%
_{\varepsilon}$, $\nu\left(  \cdot\right)  $ is the unit normal of
$\mathcal{C}_{\varepsilon}$ at $X.$ We also put
\[
\Gamma_{s}:=\left\{  X+s\nu\left(  X\right)  :X\in\mathcal{C}_{\varepsilon
}\right\}  .
\]
Note that actually $\Gamma_{s}$ depends on $\varepsilon,$ although it is not
explicit in the notation. To understand the Laplacian-Beltrami operator
$\Delta_{\Gamma_{s}}$, we need to analyze the metric on the surface
$\Gamma_{s}.$ Let $\nu_{1}=\partial_{l}\nu,\nu_{2}=\partial_{\theta}\nu,$ and
$X_{1}=\partial_{l}X,X_{2}=\partial_{\theta}X.$ Define the matrix
$B_{0}=\left[  X_{1}+s\nu_{1},X_{2}+s\nu_{2}\right]  $ and
\[
B:=\left[  X_{1}+s\nu_{1},X_{2}+s\nu_{2},v\right]  .
\]
Then the matrix of the induced metric $\mathfrak{g}$ in a tubular
neighbourhood of $\mathcal{C}$ in $\left(  l,\theta,s\right)  $ coordinate has
the form
\[
B^{T}B=\left[
\begin{array}
[c]{cc}%
B_{0}^{T}B_{0} & 0\\
0 & 1
\end{array}
\right]  .
\]
For more details, we refer to \cite{DW}.

\subsection{Proof of Theorem \ref{catenoid}}

In this part, we sketch the proof of Theorem \ref{catenoid}.

Let $h_{-1},h_{1}\in C_{loc}^{2,\alpha}\left(  \mathcal{C}_{\varepsilon
}\right)  ,$ small in certain sense. As before, define an approximate solution
$w_{h}$ in $\Omega_{h},$ which is a region trapped between $\Gamma_{-1+h_{-1}%
}$ and $\Gamma_{1+h_{1}}:$
\[
w_{h}\left(  s,l\right)  =\frac{s-g\left(  l\right)  }{1+f\left(  l\right)
},
\]
where
\[
f=\frac{h_{1}-h_{-1}}{2},g=\frac{h_{1}+h_{-1}}{2}.
\]
Still set
\[
t=\frac{s-g\left(  l\right)  }{1+f\left(  l\right)  }.
\]
The solution $u$ we are looking for will have the form $u=w_{h}+\phi.$

We have the same formulas as in Lemma $\ref{l1},$ Lemma \ref{l2} and
Lemma \ref{l3} and will not restate them in this section again.

\begin{lemma}
We have the following estimate for the Laplacian operator acting on functions
depending on $s$ and $l:$%
\[
\Delta_{\Gamma_{0}}\eta- \partial_{l}^{2}\eta=O\left(  \frac{\varepsilon
}{1+\varepsilon l}\right)  \partial_{l}\eta,
\]
and%
\[
\Delta_{\Gamma_{s}}\eta-\Delta_{\Gamma_{0}}\eta=O\left(  \frac{\varepsilon
^{2}}{\left(  1+\varepsilon l\right)  ^{2}}\right)  \partial_{l}\eta+O\left(
\frac{\varepsilon}{1+\varepsilon l}\right)  \partial_{l}^{2}\eta.
\]

\end{lemma}

\begin{proof}
The first equation has already been proved in $\left(  \ref{lap2}\right)  .$
The proof of the second equation is same as that of Lemma \ref{Lap}.
\end{proof}

Let us introduce the functional framework to work with. Let $\alpha\in\left(
0,1\right)  $ be a fixed constant.

\begin{definition}
For $\mu=0,1,2,$ $\beta\geq0,\delta>0,$ the space $\mathcal{E}_{\beta
,\mu;\delta}$ consists of those functions $\eta$ defined on $\mathcal{C}%
_{\delta}$ such that
\[
\sup_{\left\vert z\right\vert =l}\left[  \left(  1+\delta l\right)  ^{\beta
}\left\Vert \eta\right\Vert _{C^{\mu,\alpha}\left(  S_{\delta}\cap
B_{1}\left(  z\right)  \right)  }\right]  <+\infty.
\]
\end{definition}

Same as before, we also regard $\phi$ as the restriction of a function
$\mathcal{T}\left(  \phi\right)  $ on $\Xi:=\left[  -1,1\right]
\times\mathbb{[}0,+\infty),$ where $\mathcal{T}\left(  \phi\right)  $ is a
function of $t$ and $l$ defined for $\left(  t,l\right)  \in\bar{\Xi}:=\left[
-1,1\right]  \times\mathbb{R}$, even in the variable $l.$

\begin{definition}
For $\mu=0,1,2,$ $\beta\geq0$, the space $\mathcal{E}_{\beta,\mu;\ast}$
consists of those functions $\phi$ such that
\[
\left\Vert \phi\right\Vert _{\beta,\mu;\ast}:=\sup_{l\in\mathbb{R};z\in
\bar{\Xi},\left\vert z\right\vert =\left\vert l\right\vert }\left[  \left(
1+\varepsilon\left\vert l\right\vert \right)  ^{\beta}\left\Vert
\mathcal{T}\left(  \phi\right)  \right\Vert _{C^{\mu,\alpha}\left(  \bar{\Xi
}\cap B_{1}\left(  z\right)  \right)  }\right]  <+\infty.
\]

\end{definition}

Let $v\left(  \cdot\right)  $ be an even smooth function such that%
\[
v\left(  l\right)  =\left\{
\begin{array}
[c]{l}%
\left\vert l\right\vert ^{3-n},\left\vert l\right\vert >2,\\
0,\left\vert l\right\vert <1.
\end{array}
\right.
\]
The one dimensional space spanned by this function will be denoted by
$\mathcal{D}.$ Let $\bar{g}\left(  \cdot\right)  =g\left(  \frac{\cdot
}{\varepsilon}\right)  .$ If $n\geq4,$ we shall assume a priori $\bar{g}%
\in\mathcal{E}_{2n-6,2;1}\oplus\mathcal{D},$ $f\in\mathcal{E}%
_{2n-4,2;\varepsilon},$ with $\left\Vert \bar{g}\right\Vert _{\mathcal{E}%
_{2n-6,2;1}\oplus\mathcal{D}}\leq C\varepsilon,\left\Vert f\right\Vert
_{2n-4,2;\varepsilon}\leq C\varepsilon^{2}$. For notational simplicity, the
norm of $\mathcal{E}_{2n-6,2;1}\oplus\mathcal{D}$ will be denoted by
$\left\Vert \cdot\right\Vert .$ In the case $n=3,$ we assume $\bar{g}%
\in\mathcal{E}_{2,2;1}\oplus\mathcal{D},$ $f\in\mathcal{E}_{4,2;\varepsilon},$
with $\left\Vert \bar{g}\right\Vert _{\mathcal{E}_{2,2;1}\oplus\mathcal{D}%
}\leq C\varepsilon,$ $\left\Vert f\right\Vert _{4,2;\varepsilon}\leq
C\varepsilon^{2},$ and in this case, the norm of $\mathcal{E}_{2,2;1}%
\oplus\mathcal{D}$ will also be denoted by $\left\Vert \cdot\right\Vert .$

With these choice of function spaces, we can verify that $\left\Vert \Delta
w\right\Vert _{2n-4,2;\ast}\leq C\varepsilon^{2}$ if $n\geq4;$ while
$\left\Vert \Delta w\right\Vert _{4,2;\ast}\leq C\varepsilon^{2}$ if $n=3.$

Recall that the Jacobi operator on $\mathcal{C}_{\delta}$ is given by
\[
J_{\mathcal{C}_{\delta}}\left(  \eta\right)  =\Delta_{\mathcal{C}_{\delta}%
}\eta+\left\vert A\right\vert ^{2}\eta.
\]
Here $\left\vert A\right\vert ^{2}=\sum k_{i}^{2}$ is the squared norm of the
second fundamental form. Using the asymptotic behavior of $\bar{\omega},$ we
deduce $\left\vert A\right\vert ^{2}=O\left(  \frac{1}{\left(  1+l\right)
^{2n-2}}\right)  $ as $l\rightarrow+\infty.$ We need the following lemma,
which states that the Jacobi operator on the catenoid $\mathcal{C}_{1}$ is
invertible in suitable functional spaces.

\begin{lemma}
For each function $\xi\in\mathcal{E}_{2n-4,2;1},$ there is a solution $\eta
\in\mathcal{E}_{2n-6,2;1}\oplus\mathcal{D}$ such that%
\[
J_{\mathcal{C}_{1}}\left(  \eta\right)  =\xi,
\]
with%
\[
\left\Vert \eta\right\Vert \leq C\left\Vert \xi\right\Vert _{2n-4,0;1}.
\]

\end{lemma}

\begin{proof}
Detailed analysis of the Jacobi operator on the higher dimensional catenoid
can be found in \cite{Oscar}. The proof of this Lemma follows from similar
arguments there. The basic idea is using variation of parameter formula to get
the desired estimates.
\end{proof}

With this functional framework at hand, we now deal with the corresponding
linear theory for our nonlinear problem. Given functions $\gamma_{1}%
,\gamma_{-1},$ consider the problem
\begin{equation}
\left\{
\begin{array}
[c]{l}%
\partial_{t}^{2}\phi+\partial_{l}^{2}\phi=J_{\Gamma_{0}}g+\left(
\Delta_{\Gamma_{0}}f+\left\vert A\right\vert ^{2}\right)  t,\text{ in }%
\Omega_{h},\\
\phi=0\text{ on }\partial\Omega_{h},\\
\partial_{t}\phi-f=\gamma_{-1},\text{on }\Gamma_{-1+h_{-1}},\\
\partial_{t}\phi-f=\gamma_{1},\text{ on }\Gamma_{1+h_{1}}.
\end{array}
\right.  \label{clinear}%
\end{equation}

\begin{proposition}
Suppose $\gamma_{1}\pm\gamma_{-1}$ is in $\mathcal{E}_{2n-4,1;\varepsilon}$
for $n\geq4$ and in $\mathcal{E}_{4,1;\varepsilon}$ for $n=3.$ Then the system
$\left(  \ref{clinear}\right)  $ has a solution $\left(  f,\bar{g}\right)  $
such that
\begin{align*}
\left\Vert f\right\Vert _{2n-4,2;\varepsilon}  &  \leq C\left\Vert \gamma
_{1}+\gamma_{-1}\right\Vert _{2n-4,1;\varepsilon}+C\left\Vert \left\vert
A\right\vert ^{2}\right\Vert _{2n-4,1;\varepsilon},n\geq4,\\
\left\Vert f\right\Vert _{2n-4,2;\varepsilon}  &  \leq C\left\Vert \gamma
_{1}+\gamma_{-1}\right\Vert _{4,1;\varepsilon}+C\left\Vert \left\vert
A\right\vert ^{2}\right\Vert _{4,1;\varepsilon},n=3,
\end{align*}
and
\begin{align*}
\left\Vert \bar{g}\right\Vert  &  \leq C\varepsilon^{-2}\left\Vert \gamma
_{1}-\gamma_{-1}\right\Vert _{2n-4,1;\varepsilon},n\geq4,\\
\left\Vert \bar{g}\right\Vert  &  \leq C\varepsilon^{-2}\left\Vert \gamma
_{1}-\gamma_{-1}\right\Vert _{4,1;\varepsilon},\text{ }n=3.
\end{align*}

\end{proposition}

\begin{proof}
By even reflection, we can regard $\left(  \ref{clinear}\right)  $ as a
problem in $\left(  t,l\right)  \in\left[  -1,1\right]  \times\mathbb{R}.$
Take the Fourier transform%
\[
\hat{\eta}\left(  t,\xi\right)  :=\int_{\mathbb{R}}e^{-2\pi i\xi l}\eta\left(
t,l\right)  dl.
\]
It is worth mentioning that here $\xi\in\mathbb{R},$ unlike the Simons' cone
case where the Fourier transform is taken in $\mathbb{R}^{4}.$ We are lead to the problem%
\begin{equation}
\left\{
\begin{array}
[c]{l}%
\partial_{t}^{2}\hat{\phi}-\left\vert \xi\right\vert ^{2}\hat{\phi}=\left(
J_{\mathcal{C}_{\varepsilon}}g\right)  ^{\symbol{94}}+\left(  \Delta
_{\mathcal{C}_{\varepsilon}}f+\left\vert A\right\vert ^{2}\right)
^{\symbol{94}}t,\text{ }t\in\left[  -1,1\right]  ,\\
\hat{\phi}\left(  -1,\xi\right)  =\hat{\phi}\left(  1,\xi\right)  =0,\\
\partial_{t}\hat{\phi}\left(  -1,\xi\right)  -\hat{f}\left(  \xi\right)
=\hat{\gamma}_{-1}\left(  \xi\right)  ,\\
\partial_{t}\hat{\phi}\left(  1,\xi\right)  -\hat{f}\left(  \xi\right)
=\hat{\gamma}_{1}\left(  \xi\right)  \text{.}%
\end{array}
\right.  \label{clinear1}%
\end{equation}
The solution $\hat{\phi}$ of the first equation in $\left(  \ref{clinear1}%
\right)  $ can be written in the form
\[
\hat{\phi}\left(  t,\xi\right)  =\left(  J_{\mathcal{C}_{\varepsilon}%
}g\right)  ^{\symbol{94}}p_{1,\xi}\left(  t\right)  +\left(  \Delta
_{\mathcal{C}_{\varepsilon}}f+\left\vert A\right\vert ^{2}\right)
^{\symbol{94}}p_{2,\xi}\left(  t\right)  .
\]
This implies that
\[
\left\{
\begin{array}
[c]{l}%
\left(  J_{\mathcal{C}_{\varepsilon}}g\right)  ^{\symbol{94}}=\ \frac
{\hat{\gamma}_{1}\left(  \xi\right)  -\hat{\gamma}_{-1}\left(  \xi\right)
}{2p_{1,\xi}^{\prime}\left(  1\right)  },\\
\left(  \Delta_{\mathcal{C}_{\varepsilon}}f+\left\vert A\right\vert
^{2}\right)  ^{\symbol{94}}=\frac{2\hat{f}\left(  \xi\right)  +\hat{\gamma
}_{-1}\left(  \xi\right)  +\hat{\gamma}_{1}\left(  \xi\right)  }{2p_{2,\xi
}^{\prime}\left(  1\right)  }.
\end{array}
\right.
\]
Observe that $\frac{1}{p_{1,\xi}^{\prime}\left(  1\right)  }-\xi\tanh\xi$ is
real analytic and of the order $O\left(  e^{-\frac{\left\vert \xi\right\vert
}{2}}\right)  $ as $\left\vert \xi\right\vert \rightarrow+\infty.$\ According
to the proof of Lemma \ref{L2}, one need to estimate the inverse Fourier
transform of $\xi\tanh\xi\left[  \hat{\gamma}_{1}\left(  \xi\right)
-\hat{\gamma}_{-1}\left(  \xi\right)  \right]  .$ To do this, we can apply the
fact that the Fourier transform of $x\tanh\left(  \pi x\right)  $ is equal to
$-\frac{\cosh\left(  \pi\xi\right)  }{2\sinh^{2}\left(  \pi\xi\right)  }$,
which has a singularity of order $O\left(  \xi^{-2}\right)  $ near the origin.
The estimate of $f$ is similar as before.
\end{proof}

Once we have established the functional framework and the linear solvability
theory, we can proceed in the same way as the Simons' cone case.

\end{document}